\documentclass[11pt]{amsart}

\usepackage{amsmath}
\usepackage{amstext}
\usepackage{amssymb}
\usepackage{amsthm}
\usepackage{enumerate}
\usepackage{bbm}
\usepackage{comment}
\usepackage[USenglish]{babel}
\usepackage[utf8]{inputenc}
\usepackage[T1]{fontenc}

\makeatletter
\def\imod#1{\allowbreak\mkern10mu({\operator@font mod}\,\,#1)}
\makeatother

\theoremstyle{plain}
\newtheorem{thm}{Theorem}
\theoremstyle{definition}
\newtheorem{defi}[thm]{Definition}
\newtheorem{cor}[thm]{Corollary}
\newtheorem*{cor*}{Corollary}
\newtheorem*{lem*}{Lemma}
\newtheorem*{thm*}{Theorem}
\newtheorem{lem}[thm]{Lemma}
\newtheorem{question}[thm]{Question}

\newtheorem{remark}[thm]{Remark}

\newtheorem{clm}[thm]{Claim}

\DeclareMathOperator{\dom}{dom}

\DeclareMathOperator{\id}{id}
\DeclareMathOperator{\fil}{fil}

\DeclareMathOperator{\stem}{stem}
\DeclareMathOperator{\supp}{supp}

\let\split\undefined
\DeclareMathOperator{\split}{split}

\DeclareMathOperator{\FF}{FF}

\DeclareMathOperator{\Sl}{sl}

\renewcommand{\id}{\mathrm{id}}

\newcommand{\bbP}{\mathbb{P}}

\newcommand{\bbS}{\mathbb{S}}
\newcommand{\bbM}{\mathbb{M}}

\newcommand{\al}{\alpha}
\newcommand{\be}{\beta}

\newcommand{\ka}{\kappa}

\renewcommand{\phi}{\varphi}

\newcommand{\om}{\omega}

\newcommand{\calA}{\mathcal{A}}
\newcommand{\calB}{\mathcal{B}}
\newcommand{\calC}{\mathcal{C}}

\newcommand{\calE}{\ensuremath{\mathcal{E}}}
\newcommand{\calF}{\mathcal{F}}

\newcommand{\calH}{\ensuremath{\mathcal{H}}}
\newcommand{\calI}{\ensuremath{\mathcal{I}}}

\newcommand{\calP}{\mathcal{P}}

\newcommand{\calU}{\mathcal{U}}
\newcommand{\calX}{\mathcal{X}}
\newcommand{\calY}{\mathcal{Y}}

\newcommand{\concatB}{\mathbin{\rotatebox[origin=c]{90}{\scalebox{.7}{(\kern1ex)}}}}

\date{\today}
\begin{document}

\title[Higher Independence]{Higher Independence}

\author{Vera Fischer}
\address{Institute of Mathematics, University of Vienna, Kolingasse 14-16, 1090 Wien, Austria}
\email{vera.fischer@univie.ac.at}

\author{Diana Carolina Montoya}
\address{Institute of Mathematics, University of Vienna, Kolingasse 14-16, 1090 Wien, Austria}
\email{diana.carolina.montoya.amaya@univie.ac.at}

\thanks{\emph{Acknowledgments.}: The authors would like to thank the Austrian Science Fund (FWF) for the generous support through Grant Y1012-N35, as well as I4039 (Fischer) and T1100 (Montoya).}



\begin{abstract} 
We study higher analogues of the classical independence number on $\omega$.
For $\kappa$ regular uncountable, we denote by $i(\kappa)$ the minimal size of a maximal $\kappa$-independent family. We establish ZFC relations between $i(\kappa)$ and the standard higher analogues of some of the classical cardinal characteristics, e.g. $\mathfrak{r}(\kappa)\leq\mathfrak{i}(\kappa)$ and $\mathfrak{d}(\kappa)\leq\mathfrak{i}(\kappa)$. 
 
For $\kappa$ measurable, assuming that $2^\kappa=\kappa^+$ we construct a maximal $\kappa$-independent family which remains maximal after the $\kappa$-support product of $\lambda$ many copies of $\kappa$-Sacks forcing. Thus, we show the consistency of $\kappa^+=\mathfrak{d}(\kappa)=\mathfrak{i}(\kappa)<2^\kappa$. We conclude the paper with interesting open questions and discuss difficulties regarding other natural approaches to higher independence. 
\end{abstract}

\maketitle
\section{Introduction}

A family $\calA$ contained in $[\omega]^\omega$ is said to be independent if for every two finite disjoint subfamilies $\calB$ and $\calC$ the set 
$\bigcap\calB\backslash\bigcup\calC$ is infinite. We refer to such sets as boolean combinations. The least size of a maximal (under inclusion) independent family is denoted $\mathfrak{i}$. For an excellent introduction to the subject of cardinal characteristics of the continuum and definition of various characteristics we refer the reader to~\cite{blass}. 

The past decade has seen an increased volume of work regarding natural higher analogues for uncountable cardinals $\kappa$ of the classical cardinal characteristics. However, even though we already have a comparatively rich literature in this area there is very little known about analogues of the notion of independence. Even in the classical, countable setting, the independence number, and the notion of independence in general, do not seem to be that well-studied. Among the many open questions surrounding independence are the consistency of $\hbox{cof}(\mathfrak{i})=\omega$ and the consistency of $\mathfrak{i}<\mathfrak{a}$. 
A difficulty in the study of the higher independence number is the fact that it is not \emph{a priori} clear what the natural generalization of the classical independence number should be. Given an uncountable cardinal $\kappa$\footnote{In the paper, we only study the case in which $\kappa$ is regular.} one may consider subfamilies $\calA$ of $[\kappa]^\kappa$ which have the property that every boolean combination generated by strictly less than $\kappa$ many elements of $\calA$ is unbounded. That is, one may require that for every two disjoint subfamilies $\calB$ and $\calC$ of $\calA$, such that $|\calB|<\kappa$ and $|\calC|<\kappa$, the boolean combination $\bigcap\calB\backslash\bigcup\calC$ is unbounded. We refer to such families as strongly independent. A major problem presenting itself in the study of this notion of strong independence on $\kappa$ is the very existence of maximal, under inclusion, strongly independent families. Results regarding these families, together with a number of interesting open questions are included in the last section of the paper.
An earlier study of the notion of strong independence can be found 
in~\cite{KK}, where it is shown that the existence of a maximal strongly-$\omega_1$-independent family is equiconsistent with the existence of a measurable.

A more restrictive, but fruitful, approach towards the generalization of the classical notion of independence is the requirement that for a given family $\calA\subseteq[\kappa]^\kappa$ the finitely generated boolean combinations are unbounded. That is, given a family $\calA\subseteq[\kappa]^\kappa$ we say that $\calA$ is $\kappa$-independent if for every two disjoint finite subfamilies $\calB$ and $\calC$ contained in $\calA$, the set $\bigcap\calB\backslash\bigcup\calC$ is unbounded.\footnote{Clearly every strongly independent family is independent.}
The existence of a maximal under inclusion $\kappa$-independent family is provided by the Axiom of Choice and thus given an uncountable regular cardinal $\kappa$, one can define the higher independence number, denoted $\mathfrak{i}(\kappa)$, to be the minimal size of a maximal $\kappa$-independent family. A standard diagonalization argument going over  all boolean combinations, shows that $\kappa^+\leq \mathfrak{i}(\kappa)$. Classical examples of independent families of cardinality $2^\omega$ do  generalize into the uncountable and provide the existence of $\kappa$-independent  families and so of maximal $\kappa$-independent families of cardinality $2^\kappa$ (see Lemma~\ref{classic}). An example of a strongly $\kappa$-independent family of cardinality $2^\kappa$, under some additional hypothesis on $\kappa$, is provided in Lemma~\ref{strong_classic}. 

One of the main breakthroughs in the study of the classical independence number is the consistency of $\mathfrak{i}<\mathfrak{u}$, established in 1992 by S. Shelah (see~\cite{SS}). The consistency proof carries a somewhat hidden construction of a Sacks indestructible maximal independent family, that is a maximal independent family which remains maximal after the countable support product and countable support iterations of Sacks forcing. A tree version of Shelah's poset, known as party forcing, has been used in~\cite{DCJGVF} to establish the consistency of $\mathfrak{i}=\mathfrak{f}<\mathfrak{u}$, where $\mathfrak{f}$ is the free sequences number.\footnote{Partial orders and their tree versions often differ. Good example is given by the poset for adding an infinitely often equal real and its tree version. The former adds a Silver real and so kills all $p$-points (see~\cite{DCOG}), while the latter preserves $p$-points (see~\cite{MGHJSS}).} For recent studies on Sacks indestructible, co-analytic maximal independent families see~\cite{JBVFYK}, as well as \cite{CCVFOGJS, JBVFCS, CS}.  In this paper, we prove:

\begin{thm*} Let $\kappa$ be  a measurable cardinal and let 
$2^\kappa=\kappa^+$. Then there is a maximal $\kappa$-independent family which remains maximal after the $\kappa$-support product of $\lambda$-many copies of $\kappa$-Sacks forcing.
\end{thm*}

The existence of this indestructible maximal $\kappa$-independent family is closely related to the properties of a normal measure $\calU$ on $\kappa$. With the indestructible family $\calA$, we associate a $\kappa^+$-complete filter $\fil_{<\omega,\kappa}(\calA)$ which is properly contained in $\calU$ and its elements meet every boolean combination on an unbounded set. The properties of this filter capture to a great extent the indestructibility of the associated independent family. 

For readers familiar with the countable setting, we will draw a more detailed comparison. As developed originally in the work of Shelah~\cite{SS} and later analyzed for example in~\cite{DCJGVF}, an independent family which is maximal in a strong sense and whose density filter (a similarly to $\fil_{<\omega,\kappa}(\calA)$ associated filter) is selective (which means both a $P$-set and $Q$-set) is indestructible by countable support products and iterations of Sacks forcing. While $\fil_{<\omega,\kappa}(\calA)$ is a $\kappa$-$P$-set, in the sense that every subfamily of cardinality $\leq \kappa$ has a pseudo-intersection in the filter (see Definition~\ref{def_Pset_Qset} and Lemma \ref{P-set}), the role of the $Q$-set property from the countable setting is taken by the fact that for every ground model strictly increasing function in $^\kappa\kappa$ there is a set in the filter whose enumeration function grows faster than the given function (see Lemma~\ref{analogue_Qset} and Corollary~\ref{closure_club}).\footnote{In the countable setting the filter in question is generated by a tower, see for example \cite{JBVFYK} and so the filter remains a $P$-set throughout an iteration of Sacks forcing. For our current argument, which only deals with products of $\kappa$-Sacks forcing, the fact that $\fil_{<\omega,\kappa}(\calA)$ is a $\kappa$-$P$-set is sufficient.}
In both, the countable and the uncountable setting a strengthening of the maximality of the corresponding maximal independent family plays an important role. In the countable setting this strengthening is known as dense maximality, a property which originally appears in~\cite{MGSS}. The $\kappa$-maximal independent family which we construct is densely maximal in a similar sense (see Definition~\ref{def.dense.max}). Moreover we make an explicit use of an equivalent characterisation of dense maximality given in Lemma \ref{equiv}, characterization which plays a key role in our main theorem. An analogue to the countable setting of the overall approach, which we take in this paper can be found in the more recent studies \cite{JBVFCS, CS, CCVFOGJS}.  Note that an analogue of the equivalent characterization given in Lemma~\ref{equiv} implicitly appears in \cite{SS}.


Finally, the existence of a $\kappa$-mad family, which remains maximal after an arbitrarily long $\kappa$-supported product of $\kappa$-Sacks reals is a straightforward generalization of the classical case. Moreover, if $\mathfrak{d}(\kappa)=\kappa^+$ then $\mathfrak{a}(\kappa)=\kappa^+$ 
(see~\cite{BHZ} and~\cite{DRSS}). Thus our result leads to the following statement:

\begin{thm*}
	Let $\kappa$ be a measurable cardinal and $2^\kappa=\kappa^+$. Then there is a cardinal preserving generic extension in which 
	$$\mathfrak{a}(\kappa)=\mathfrak{d}(\kappa)=\mathfrak{r}(\kappa)=\mathfrak{i}(\kappa)=\kappa^+<2^\kappa.$$
\end{thm*}

One of the very interesting open questions regarding the classical independence number is the consistency of $\mathfrak{i}<\mathfrak{a}$. As a very partial result towards this question we obtain the following:

\begin{cor*}
	Let $\kappa$ be regular uncountable. If $\mathfrak{i}(\kappa)=\kappa^+$ then $\mathfrak{a}(\kappa)=\kappa^+$.
\end{cor*}

\noindent
{\emph{Structure of the paper:}} In Section~\ref{higher_ind} we define a notion of independence at $\kappa$, for $\kappa$ arbitrary infinite cardinal
and define the cardinal number $\mathfrak{i}(\kappa)$ for $\kappa$ regular uncountable. In Section~\ref{the_poset}, given a measurable cardinal $\kappa$, witnessed by a normal measure $\calU$ and working under the hypothesis that $2^\kappa=\kappa^+$, we define a $\kappa^+$ closed poset $\bbP_\calU$ which adjoins a maximal $\kappa$-independent family, which we denote $\calA_G$.\footnote{Using the normal measure $\calU$ and the hypothesis $2^\kappa=\kappa^+$ one can alternatively use the properties of the poset $\bbP_\calU$ to construct a family $\calA$ having all essential properties of $\calA_G$ using a transfinite recursion of length $\kappa^+$.} In Section~\ref{density_ideal_section} we study the properties of an ideal on $\kappa$, to which we refer as density ideal
and denote $\id_{<\omega,\kappa}(\calA_G)$, which is contained in the dual ideal of $\calU$ and which naturally captures crucial properties of the independent family $\calA_G$. In section~\ref{partition_properties_section}, we show that the dual filter of this ideal, denoted $\fil_{<\omega,\kappa}(\calA_G)$ is a $\kappa$-P-set. 
In Section~\ref{dense_maximality_section} we show that the family $\calA_G$ is densely maximal in a natural sense and characterize dense maximality in terms of properties of the density ideal. Section~\ref{outer_hull} introduces the concepts of preprocessed conditions and outer hulls necessary for the proofs in the last section. In Section~\ref{Sacks_indestructibility_section} we prove our main theorem, by showing that the densely maximal $\kappa$-independent family $\calA_G$ 
remains maximal after the $\kappa$-support product of $\lambda$ many copies of $\kappa$-Sacks forcing. We conclude the paper with some open questions and an appendix, discussing the notion of strong independence.

\section{The Higher Independence Number}\label{higher_ind}

In the following we set to define a higher analogue of the notion of independent families on $\omega$ for the special case in which the boolean combinations are finitely generated.

\begin{defi} Let $\kappa$ be a regular uncountable cardinal and let $\FF_{<\omega,\kappa}(\calA)$ be the set of all finite partial functions with domain included in $\calA$ and range the set $\{0,1\}$. For each $h\in\FF_{<\omega,\kappa}(\calA)$ let $\calA^h=\bigcap\{\calA^{h(A)}:A\in\dom(h)\}$
where $\calA^{h(A)}=A$ if $h(A)=0$ and $\calA^{h(A)}=\kappa\backslash A$
if $h(A)=1$.  We refer to sets of the form $\calA^h$ as boolean combinations. 
\end{defi}

With this we can state the definition of $\kappa$-independence. For a discussion of the most general definition in which the boolean combinations are generated by arbitrarily large subfamilies of the given family, see~\cite{VFME}.

\begin{defi} $ $
	\begin{enumerate}
		\item A family $\calA\subseteq [\kappa]^\kappa$ is said to be $\kappa$-independent if for each $h\in\FF_{<\omega,\kappa}(\calA)$ the set $\calA^h$ is unbounded. It is said to be a maximal $\kappa$-independent family if it is $\kappa$-independent and maximal under inclusion. 
		\item 	The least size of a maximal $\kappa$-independent family
	is denoted $\mathfrak{i}(\kappa)$. 
\end{enumerate}
\end{defi}

\begin{remark}
	For $\kappa=\omega$ the above notions coincides with the classical notions of independence on $[\omega]^\omega$ and $\mathfrak{i}(\kappa)=\mathfrak{i}$, where $\mathfrak{i}$ is the classical independence number.  
\end{remark}

\begin{lem}\label{classic}
	Let $\kappa$ be a regular infinite cardinal. Then
	\begin{enumerate}
		\item Every $\kappa$-independent family is contained in a maximal $\kappa$-independent family.
		\item $\kappa^+\leq\mathfrak{r}(\kappa)\leq\mathfrak{i}(\kappa)$
		\item There is a maximal $\kappa$-independent family of cardinality $2^\kappa$.
		\item $\mathfrak{d}(\kappa)\leq\mathfrak{i}(\kappa)$. 
	\end{enumerate}
\end{lem}
\begin{proof} Since the increasing union of a collection of $\kappa$-independent families is $\kappa$-independent, by the Axiom of Choice every $\kappa$-independent family is contained in a maximal one.  Note that if
$\calA$ is a maximal $\kappa$-independent family, then the set of 
boolean combinations $\{\calA^h: h\in\FF_{<\omega,\kappa}(\calA)\}$ is not split and so $\mathfrak{r}(\kappa)\leq |\calA|$. For a construction of a $\kappa$-independent family of cardinality $2^\kappa$, see ~\cite[Theorem 4.2]{SG}. Finally, the proof that $\mathfrak{d}(\kappa)\leq\mathfrak{i}(\kappa)$ follows closely the proof of the classical case, i.e. $\mathfrak{d}\leq\mathfrak{i}$ (see~\cite{LH}).
\end{proof}

One of the most interesting open questions, regarding the classical cardinal characteristics is the consistency of $\mathfrak{i}<\mathfrak{a}$. By the last item of the above theorem and the fact that if  $\mathfrak{d}(\kappa)=\kappa^+$ implies that $\mathfrak{a}(\kappa)=\kappa^+$ (see~\cite{BHZ} and~\cite{DRSS}), we obtain the following:

\begin{cor} Let $\kappa$ be a regular uncountable cardinal. Then if $\mathfrak{i}(\kappa)=\kappa^+$ then $\mathfrak{a}(\kappa)=\kappa^+$.
\end{cor}

\section{Adjoining a maximal $\kappa$-independent family}\label{the_poset}

In this section, we provide a partial order which adjoins a maximal $\kappa$-independent family, family which we will later show to be indestructible by products of $\kappa$-Sacks forcing. 

Let $\kappa$ be a measurable cardinal and $\calU$ a normal measure on $\kappa$. 

\begin{defi} Let $\mathbb{P}_\calU$ be the poset of all pairs $(\calA,A)$ where  $\calA$ is a $\kappa$-independent family of cardinality $\kappa$ and $A\in\calU$ has the property that $\forall h\in\FF_{<\omega,\kappa}(\calA)$ the set 
	$\calA^h\cap A$ is unbounded. The extension relation is defined as follows: $(\calA_1,A_1)\leq (\calA_0,A_0)$ if and only if $\calA_1\supseteq \calA_0$ and $A_1\subseteq^* A_0$.\footnote{Throughout $A\subseteq^* B$ means $|A\backslash B|<\kappa$.}
\end{defi}

\begin{lem} Assume $2^\kappa=\kappa^+$. Then $\bbP_\calU$ is $\kappa^+$-closed and $\kappa^{++}\hbox{-cc}$.
\end{lem}
\begin{proof} Let $\{(\calA_i,A_i)\}_{i\in\kappa}$ be a decreasing sequence in $\bbP_\calU$. We can assume that $\{A_i\}_{i\in\kappa}$ is strictly decreasing, i.e for each $i>j$ we have $A_j\subseteq A_i$. Then $\calA=\bigcup_{i\in\kappa}\calA_i$ is an independent family of cardinality $\kappa$ and the diagonal intersection $A'=\Delta_{i\in\kappa}A_i\in\calU$. 
	
Now, for each $i\in\kappa$, let $\{h_{i,j}\}_{j\in\kappa}$ enumerate $\FF_{<\omega,\kappa}(\calA)$. Recursively we will define a set $A''=\{k_{i,l,m}\}_{l,m<i; i<\kappa}$ which is a pseudo-intersection of $\{A_i\}_{i\in\kappa}$ and which meets every boolean combination $\calA^h$ for $h\in\FF_{<\omega,\kappa}(\calA)$ on an unbounded set.  Then $A=A'\cup A''$ is an element of $\calU$ and $(\calA,A)\in\bbP_\calU$ is a common extension of $\{(\calA_i,A_i)\}_{i\in\kappa}$.
	
{\emph{Construction of $A''$:}} At step $i$ pick $k_{i,m,l}\in A_i\cap \calA_i^{h_{m,l}}$ for each $m,l<i$. Then in particular $k_{i,m,l}\in A_m$ for each $m\leq i$ and $k_{i,m,l}\in\calA_m^{h_{m,l}}$ for each $m,l < i$.
Take $A''=\{k_{i,m,l}\}_{m,l<i;i<\kappa}$. Then $A''$ meets every boolean combination on an unbounded set and is a pseudo-intersection. Fix $\gamma\in\kappa$. Then for all $\xi$ such that $\xi>\gamma$ and all $m,l<\xi$ we have that $k_{\xi,l,m}\in A_\xi\subseteq A_\gamma$. Thus 
$A''\backslash A_\gamma\subseteq \{k_{\xi,l,m}\}_{\xi<\gamma; l,m<\xi}$, which is a bounded set. 

The poset has the $\kappa^{++}$-cc, because $|\bbP_\calU|=\kappa^+$. Indeed, $\Big|\big[[\kappa]^\kappa\big]^\kappa\Big|=\kappa^+$.
\end{proof}

\begin{lem}\label{small} If $(\calA,A)\in\bbP_\calU$, then there is $B\notin\calA$ such that $B\subseteq A$ and $(\calA\cup\{B\},A)\leq (\calA,A)$.
\end{lem}	
\begin{proof}
Let $\{h_i\}_{i\in\kappa}$ be a fixed enumeration of $\FF_{<\omega,\kappa}(\calA)$. Since $\calA^{h_0}\cap A$ is unbounded, we can find distinct $k_{0,0}$, $k_{0,1}$ in  $\calA^{h_0}\cap A$. Suppose we have defined $\{k_{i,j}: i\in\gamma, j\in 2\}$ distinct. Since $\calA^{h_\gamma}\cap A$ is unbounded, we can find distinct $k_{\gamma,0}, k_{\gamma,1}$ in $(\calA^{h_\gamma}\cap A)\backslash \{k_{i,j}: i\in\gamma, j\in 2\}$. Finally, take $B=\{k_{i,0}\}_{i\in\kappa}$. Clearly $B\subseteq A$ and $\calA\cup\{B\}$ is independent. To verify the latter note that for each  $h\in\FF_{<\omega,\kappa}(\calA)$ there are unboundedly many $h_i\supseteq h$.  Then for unboundedly many $i\in\kappa$, $k_{i,0}\in \calA^{h_i}\cap B\subseteq \calA^h\cap B$ and $k_{i,1}\in \calA^{h_i}\backslash B\subseteq \calA^h\backslash B$.  
\end{proof}

\begin{cor}\label{maximality} Let $G$ be $\bbP_\calU$-generic filter. Then $\calA_G=\bigcup\{\calA: \exists A\in\calU \hbox{ with }(\calA,A)\in G\}$ is a $\kappa$-maximal independent family.
\end{cor}
\begin{proof} Suppose $X\in [\kappa]^\kappa\setminus\calA_G$ and $\calA_G\cup\{X\}$ is independent. Take $(\calA,A)\in G$ such that 
	$$(\calA,A)\Vdash ``\calA_G\cup\{X\}\hbox{ is independent and }X\notin \calA_G".$$
Since $\bbP_\calU$ is $\kappa^+$-closed, the set $X$ belongs to the ground model. Now, if for each $h\in\FF_{<\omega,\kappa}(\calA)$ the intersections $\calA^h\cap X\cap A$ and $\calA^h\cap A\cap X^c$ are unbounded, then $(\calA\cup\{X\},A)\leq (\calA,A)$ and $$(\calA\cup\{X\},A)\Vdash ``X\in\calA_G",$$ which is a contradiction. Therefore there is $h\in\FF_{<\omega,\kappa}(\calA)$ such that either $\calA^h\cap A\cap X$ or $\calA^h\cap A\cap X^c$ is bounded. However, by Lemma~\ref{small}, there is $B\notin \calA$ such that $B\subseteq A$ and $(\calA\cup\{B\},A)\leq (\calA,A)$. But then, 
	$$(\calA\cup\{B\},A)\Vdash ``\exists h\in\FF_{<\omega,\kappa}(\calA_G)\hbox{ such that }\calA^h_G\cap X\hbox{ or }\calA^h_G\setminus X \hbox{ is bounded}."$$
	Therefore $(\calA\cup\{B\}, A)\Vdash ``\calA_G\cup\{X\}\hbox{ is not independent}"$, which is a contradiction.
\end{proof}

\begin{remark} Given a measurable cardinal $\kappa$ and a normal measure $\calU$ on $\kappa$, whenever $\calA=\calA_G$ is the generic maximal $\kappa$-independent family given by a $\mathbb{P}_\calU$-generic  filter $G$, we will say that $\calA$ is $\calU$-supported.
\end{remark}

\section{The Density Ideal}\label{density_ideal_section}

The density ideal (see~\cite{VFDM1}) plays an important roles in describing the properties of maximal independent families on $\omega$. A higher analogue of this notion will play an equally important role in the study of maximal $\kappa$-independent families indestructible by $\kappa$-Sacks forcing.

\begin{defi}\label{def_density_ideal}
Let $\calA$ be a $\mathcal{U}$-supported independent family. The density ideal $\id_{<\omega,\kappa}(\calA)$ is the ideal of all $X\in \calU^*$, where $\calU^*$ is the dual ideal of $\calU$, such that $\forall h\in\FF_{<\omega,\kappa}(\calA)$ there is $h'\in\FF_{<\omega,\kappa}(\calA)$ such that $h'\supseteq h$ and $\calA^{h'}\cap X=\emptyset$. 
\end{defi}	

\begin{lem}\label{monotonicity} $ $
	\begin{enumerate}
		\item If $\calA$ be an independent family, then $\id_{<\omega, \kappa}(\calA)$ is an ideal. 
		\item If $\calA_0,\calA_1$ are independent families such that $\calA_0\subseteq \calA_1$, then $\id_{<\omega,\kappa}(\calA_0)\subseteq\id_{<\omega,\kappa}(\calA_1)$.
	\end{enumerate}
\end{lem}
\begin{proof}
	To prove item $(1)$ above consider any $X_0$ and $X_1$ in $\id_{<\omega,\kappa}(\calA)$. Fix any $h\in\FF_{<\omega,\kappa}(\calA)$. Then there is $h_0\supseteq h$ such that $\calA^{h_0}\cap X_0=\emptyset$ and there is $h_1\supseteq h_0$ such that $\calA^{h_1}\cap X_1=\emptyset$. But then $h_1\supseteq h$ and $\calA^{h_1}\cap (X_0\cup X_1)=\emptyset$. Clearly, $\id_{<\omega,\kappa}(\calA)$ is closed under subsets and thus $\id_{<\omega,\kappa}(\calA)$ is an ideal.
	
	To prove item $(2)$ consider any $X\in\id_{<\omega,\kappa}(\calA_0)$. Let $h\in\FF_{<\omega,\kappa}(\calA_1)$. Then $h'=h\upharpoonright \calA_0\in \FF_{<\omega,\kappa}(\calA_0)$ and by hypothesis there is $h_0$ in $\FF_{<\omega,\kappa}(\calA_0)$ extending $h'$ such that $\calA_0^{h_0}\cap X=\emptyset$. Let $h_1=h_0\cup h\upharpoonright (\calA_1\backslash\calA_0)$. Then 
	$\calA_1^{h_1}\cap X\subseteq \calA_0^{h_0}\cap X$ and so $\calA_1^{h_1}\cap X=\emptyset$. 
\end{proof}
	
\begin{remark} Note that $\id_{<\omega,\kappa}(\calA)$ is not necessarily $\kappa$-complete.
\end{remark}
	
\begin{lem}\label{generic_density_ideal} $\Vdash_{\mathbb{P_\calU}}\id_{<\omega,\kappa}(\calA_G)=\bigcup\{\id_{<\omega,\kappa}(\calA):\exists A (\calA,A)\in G\}$.
\end{lem}
\begin{proof} To see 
	$\Vdash_{\bbP_\calU}\bigcup\{\id_{<\omega,\kappa}(\calA):\exists A(\calA,A)\in G\}\subseteq \id_{<\omega,\kappa}(\calA_G)$ consider any $\bbP_\calU$-generic filter $G$. In $V[G]$ we have 
	$\calA_G=\bigcup\{\calA:\exists A(\calA,A)\in G\}$. Now for all $(\calA,A)\in G$, by Lemma~\ref{monotonicity}.(2), $\id_{<\omega,\kappa}(\calA)\subseteq\id_{<\omega,\kappa}(\calA_G)$.
	Therefore $\bigcup\{\id_{<\omega,\kappa}(\calA):\exists A(\calA,A)\in G\}\subseteq \id_{<\omega,\kappa}(\calA_G)$.
		
	The fact that $\Vdash_{\mathbb{P_\calU}}\id_{<\omega,\kappa}(\calA_G)\subseteq\bigcup\{\id_{<\omega,\kappa}(\calA):\exists\calA (\calA,A)\in G\}$ follows from the 
	$\kappa^+$-closure of $\bbP_\calU$. Consider any $p=(\calA,A)\in G$ and a $\bbP_\calU$-name $\dot{X}$ for a subset of $\kappa$ such that $p\Vdash \dot{X}\in\id_{<\omega,\kappa}(\calA_G)$. Fix $h\in\FF_{<\omega,\kappa}(\calA)$. Then 
	$$p\Vdash\exists h'\in\FF_{<\omega,\kappa}(\calA_G)(h\subseteq h'\hbox{ and }\calA_G^{h'}\cap X=\emptyset).$$ 
	Thus there is $(\calA',A')\leq (\calA,A)$ such that $h'\in\FF_{<\omega,\kappa}(\calA')$, $h'\supseteq h$ and $\calA^{h'}\cap X=\emptyset$. Proceed inductively to construct a decreasing sequence $\{(\calA_i,A_i)\}_{i\in\kappa}$ of conditions below $p$ such that if $\calA_\kappa=\bigcup_{i\in\kappa}\calA_i$ then for all 
	$h\in\FF_{<\omega,\kappa}(\calA_\kappa)$ there is $h'\in\FF_{<\omega,\kappa}(\calA_\kappa)$ extending $h$ and such that 
	$\calA^{h'}\cap X=\emptyset$. Thus $X\in\id_{<\omega,\kappa}(\calA_\kappa)$. By the $\kappa^+$-closure of $\bbP_\calU$, there is $p'=(\calB,B)\in\bbP_\calU$ which is an extension of	all $(\calA_i,A_i)$. Thus $X\in\id_{<\omega,\kappa}(\calB)$, $p'\leq p$ and $$p'\Vdash \dot{X}\in \bigcup\{\id_{<\omega,\kappa}(\calA): \exists A(\calA,A)\in G\}.$$
\end{proof}

\begin{lem}\label{small_sets}
		Let $(\calA,A)\in\mathbb{P}_\calU$ and let $X\in\id_{<\omega,\kappa}(\calA)$. Then $(\calA,A\backslash X)\in\mathbb{P}_\calU$. 	
\end{lem}
\begin{proof} It is sufficient to show that for each $h\in\FF_{<\omega,\kappa}(\calA)$ the set $\calA^h\cap(A\backslash X)$ is unbounded. Fix $h\in\FF_{<\omega,\kappa}(\calA)$. Since $X\in\id_{<\omega,\kappa}(\calA)$ there is $h'\supseteq h$, $h'\in\FF_{<\omega,\kappa}(\calA)$ extending $h$ such that $\calA^{h'}\cap X=\emptyset$. Thus $\calA^{h'}\subseteq\kappa\backslash X$. However 
	$$\calA^{h'}\cap A=(\calA^{h'}\cap A\cap X)\cup (\calA^{h'}\cap A\cap X^c).$$
Thus $\calA^{h'}\cap A=\calA^{h'}\cap A\cap X^c$ is unbounded. Therefore 
$(\calA,A\backslash X)$ is a condition.
\end{proof}

\begin{cor}\label{ideal_generators} Let $G$ be a $\bbP_\calU$-generic filter. Then in $V[G]$ the density ideal $\id(\calA_G)$ is generated by $\{\kappa\backslash A:\exists\calA(\calA,A)\in G\}$. That is 
	$$\Vdash_{\bbP_\calU}\id_{<\omega,\kappa}(\calA_G)=<\{\kappa\backslash A:\exists \calA(\calA,A)\in G\}>.$$
\end{cor}
\begin{proof}
Let $G$ be a $\bbP_\calU$-generic filter. By Lemma~\ref{generic_density_ideal},  $\id_{<\omega,\kappa}(\calA_G)=\bigcup\{\id_{<\omega,\kappa}(\calA):\exists A(\calA,A)\in G\}$.  Let $\calI_G$ be the ideal generated by 
	$\{\kappa\backslash A:\exists \calA(\calA,A)\in G\}$.
	
First we will show that $\id_{<\omega,\kappa}(\calA_G)\subseteq\calI_G$. 
Let $X\in\id_{<\omega,\kappa}(\calA_G)$. Thus there is $(\calA,A)\in G$ such that $X\in\id_{<\omega,\kappa}(\calA)$. However the set 
$D_X=\{(\calB,B)\in\bbP_\calU: X\cap B=\emptyset\}$ is dense below $(\calA,A)$
(indeed, if $(\calB,B)\leq (\calA,A)$ then $X\in\id_{<\omega,\kappa}(\calB)$ and by Lemma~\ref{small_sets} $(\calB,B\backslash X)\leq (\calB,B)$) and so there is $(\calB,B)\in G$ such that $X\cap B=\emptyset$. That is $X\subseteq \kappa\backslash B$ and so $X\in\calI_G$. 

To show that $\calI_G\subseteq\id_{<\omega,\kappa}(\calA_G)$, consider any $X\in\calI_G$. Then there is a finite set of conditions $\{(\calA_i,A_i)\}_{i\in n}$ in $G$ such that $X\subseteq \bigcup_{i\in n}\kappa\backslash A_i=\kappa\backslash\bigcap_{i\in n}A_i$. Note that 
$(\calB,B)\in G$, where $(\calB,B)=(\bigcup_{i\in n} \calA_i, \bigcap_{i\in n} A_i)$. Thus $X\subseteq \kappa\backslash B$.  Fix any $h\in \FF_{<\omega,\kappa}(\calA_G)$. Then there is $(\calC,C)\in G$ such that $h\in\FF_{<\omega,\kappa}(\calC)$. Take $(\calE,E)\in G$ which is a common extension of $(\calB,B)$ and $(\calC,C)$. Then $(\calE, E)\leq (\calC, B)$ and so in particular $(\calC,B)\in G$. However the set $H_B=\{(\calC',C'):\exists Y\in\calC'(Y\subseteq B)\}$ is dense below $(\calC,B)$ (apply Lemma~\ref{small}) and so there is $(\calC',C')\in G$ such that for some $Y\in \calC'$, $Y\subseteq B$. Then $h'=h\cup\{(Y,0)\}\in\FF_{<\omega,\kappa}(\calA_G)$ and $\calA_G^{h'}\cap X=\emptyset$. Thus $X\in \id_{<\omega,\kappa}(\calA_G)$.
\end{proof}

\section{The Density Filter}\label{density_filter_section}

Of particular interest for our investigations will be the dual filter of the density ideal. Note that the density filter plays an important role in the original work of~\cite{SS} on the relative consistency of $\mathfrak{i}<\mathfrak{u}$, from which the existence of a Sacks indestructible maximal independent family can be extracted (see also~\cite{VFDM1}).

\begin{remark}\label{generators_remark}
	Let $G$ be $\bbP_\calU$-generic, let $\calF_G=\{A:\exists\calA\hbox{ such that }(\calA,A)\in G\}$ and let  $\fil_{<\omega,\kappa}(\calA_G)$ be the dual filter of $\id_{<\omega,\kappa}(\calA_G)$. By Corollary~\ref{ideal_generators}, $\fil_{<\omega, G}(\calA_G)$ is generated by $\calF_G$.
\end{remark}

\begin{lem}\label{two_partitions}
Let $(\calA,A)\in\bbP_\calU$, $Y\in [\kappa]^\kappa$ and $h\in\FF_{<\omega,\kappa}(\calA)$. Then there is $h^*\supseteq h$ in 
$\FF_{<\omega,\kappa}(\calA)$ and $B\subseteq A$ such that 
$(\calA,B)\leq (\calA,A)$ and $\calA^{h^*}\cap B$ is contained either in $Y$, or in $\kappa\backslash Y$.
\end{lem}
\begin{proof} If there is $h'$ extending $h$ such that $\calA^{h'}\cap A\cap Y$ is bounded, then $\calA^{h'}\cap A=^* \calA^{h'}\cap A\cap(\kappa\backslash Y)$ and so for all $h''\supseteq h'$ the set 	$\calA^{h''}\cap A\cap (\kappa\backslash Y)$ is unbounded. Then take 
$B=(\calA^{h'}\cap A\cap (\kappa\backslash Y))\cup (A\backslash \calA^{h'})$.
Then $B=^* A$ and so $B\in \calU$, $(\calA,B)$ is as desired.

If there is $h'\supseteq h$ such that $\calA^{h'}\cap A\cap (\kappa\backslash Y)$ is bounded, then $\calA^{h'}\cap A=^* \calA^{h'}\cap A\cap Y$ and so for all $h''\supseteq h'$ the set $\calA^{h''}\cap A\cap Y$ is unbounded. Then take $B=(\calA^{h'}\cap A\cap Y)\cup (A\backslash \calA^{h'})$. Then $B=^* A$ and so $B\in\calU$, and the condition $(\calA,B)$ is as desired. 

Suppose, none of the above two cases holds. Thus for every $h'\supseteq h$, the sets $\calA^{h'}\cap A\cap Y$ and $\calA^{h'}\cap A\cap (\kappa\backslash Y)$ are unbounded. Then each of the sets 
$B_0=(\calA^h\cap A\cap Y)\cup (A\backslash\calA^h)$ and $B_1=(\calA^h\cap A\cap (\kappa\backslash Y))\cup (A\backslash\calA^h)$ meets every boolean combination $\calA^{h'}$ for $h'\in\FF_{<\omega,\kappa}(\calA)$ on an unbounded set. Thus if $A\backslash\calA^h\in \calU$, both $B$ and $B'$ are as desired. Suppose $A\backslash\calA^h\notin \calU$. Then $A\cap \calA^h\in \calU$ and so either $A\cap \calA^h\cap Y$ or $A\cap\calA^h\cap (\kappa\backslash Y)$ is in the normal measure. We can chose appropriately.
\end{proof}

\begin{cor}\label{density_partitions} 
Let $\calE=\{Y,\kappa\backslash Y\}$ be a partition. Then the set of  $(\calA,A)\in\bbP_\calU$ such that for each $h\in\FF_{<\omega,\kappa}(\calA)$ there is $h'\supseteq h$ in $\FF_{<\omega,\kappa}(\calA)$ with the property that $\calA^{h'}$ is either contained in $Y$, or in $\kappa\backslash Y$ is dense in $\bbP_\calU$.
\end{cor}
\begin{proof}
Consider an arbitrary $(\calA,A)\in\bbP_\calU$. Fix $h_0\in\FF_{<\omega,\kappa}(\calA)$. Then there is $A_0\subseteq A$ such that $(\calA,A_0)\leq (\calA,A)$ and there is $h_1\in\FF_{<\omega,\kappa}(\calA)$ extending $h_0$ and $B\subseteq A$ such that $\calA^{h_1}\cap B$ is  contained either in $Y$, or in $\kappa\backslash Y$. However, by Lemma~\ref{small} there is $B_0\subseteq B$ such that $(\calA\cup\{B_0\},B)\leq (\calA, B)$. Then extend $h_1$ to $h_1'=h_1\cup\{(B_0,0)\}$ and note that $h_1'\in\FF_{<\omega,\kappa}(\calA_1)$, where $\calA_1=\calA\cup\{B_0\}$,
and that $\calA_1^{h_1'}$ is either contained in $Y$ or in $\kappa\backslash Y$. Proceed inductively and use the fact that $\bbP_\calU$ is $\kappa^+$-closed. 
\end{proof}

\begin{defi}\label{def_Pset_Qset} Let $\calF\subseteq [\kappa]^\kappa$. We say that $\calF$ is a $\kappa\hbox{-P-set}$ if every $\calH\subseteq \calF$ of cardinality $\leq\kappa$ has a pseudo-intersection in $\calF$.
\end{defi}

\begin{lem}\label{P-set} Let $G$ be a $\bbP_\calU$-generic filter. Then  $\calF_G$ is a $\kappa\hbox{-P-set}$. 
\end{lem}
\begin{proof} Suppose $\calF_G$ is not a $\kappa\hbox{-P-set}$. Thus there is $p\in\bbP_\calU$ such that 
	$$p\Vdash\exists\calH\in[\calF_G]^\kappa\hbox{ s.t. }\forall F\in\calF_G\exists H\in\calH (F\not\subseteq^* H).$$
Fix $G$ a $\bbP_\calU$-generic filter such that $p\in G$. Since $\bbP_\calU$ is $\kappa^+$-closed, we can find $\calH'=\{A_i\}_{i\in\kappa}$ in the ground model witnessing the above property. For each $i\in \kappa$, let $\calA_i$ be such that $(\calA_i, A_i)\in G$. We can assume that $\tau=\{(\calA_i,A_i)\}_{i\in\kappa}$ is decreasing and that $(\calA_0,A_0)\leq p$. Now, take $q=(\calA,A)$ in $\bbP_\calU$ to be a common lower bound of $\tau$. Then $q\leq p$ and $q$ forces that $A$ is a pseudo-intersection of $\calH'$, which is a contradiction.   	
\end{proof}

\subsection{Increasing Functions and the Density Filter}\label{partition_properties_section}

In the countable setting a key feature of the Sacks indestructible maximal independent family appearing in \cite{SS} is the fact that the associated density filter is a $Q$-set. The existence of sufficiently fast growing sets in $\fil_{<\om, \kappa}(\calA_G)$, which we discuss in this section, will be of vital importance for our main result.

If $E\subseteq\kappa$ is an unbounded set and for each $\alpha\in \kappa$ let $s_E(\alpha)=\min\{\beta\in E: \beta>\alpha\}$.

\begin{lem}\label{analogue_Qset} Let $f\in V\cap{^\kappa\kappa}$ be a strictly increasing function and let $(\calA,A)\in\mathbb{P}_\calU$. Then there is $A^*\subseteq A$ such that $(\calA,A^*)\leq (\calA,A)$ and if $\{a(i)\}_{i<\kappa}$ is the increasing enumeration of $A^*$ then $f(a(i))<a(i+1)$ for all $i$.
\end{lem}
\begin{proof} Let $C_f=\{\xi<\kappa: \forall\zeta<\xi (f(\zeta)<\xi)\}$. Thus $C_f$ is a club and so $C_f\in\calU$. Then $E=A\cap C_f\in \calU$. 
Let $\{h_i\}_{i<\kappa}$ be an enumeration of the elements of $\FF_{<\omega,\kappa}(\calA)$ such that each element occurs unboundedly often. The set $A^*$ will be constructed as the union of an increasing sequence 
	$\{B_\xi\}_{\xi<\kappa}$ of subsets of $A$. 
	
Let $B_0=\emptyset$. If $\calA^{h_0}\cap E\neq\emptyset$, let $a_0=\min \calA^{h_0}\cap E$. Otherwise, take $a_0=\min \calA^{h_0}\cap A$. 
Let $$B_1=\{a_0\}\cup (E\cap s_E(f(a_0)))\cup \{s_E(f(a_0))\}.$$
Suppose we have defined $B_\xi$. If $(\calA^{h_{\xi+1}}\cap E)\backslash (B_\xi\cup\{\sup B_\xi\})\neq\emptyset$, let $a_{\xi+1}=\min ((\calA^{h_{\xi+1}}\cap E)\backslash (B_\xi\cup\{\sup B_\xi\}))$. Otherwise, let $a_{\xi+1}=\min\{a\in \calA^{h_{\xi+1}}\cap A: a>\sup B_\xi\}$. Let 
$$B_{\xi+1}= B_\xi\cup\{a_{\xi+1}\}\cup (E\cap
s_{E}(f(a_{\xi+1})))\cup\{s_E(f(a_\xi))\}.$$
Now, suppose $\xi$ is a limit and for all $\zeta<\xi$, the set $B_\zeta$ has been defined. Take $B_\xi^*=\bigcup_{\zeta<\xi} B_\zeta$. 
If $(\calA^{h_\xi}\cap E)\backslash (B_\xi^*\cup\{\sup B_\xi^*\})\neq\emptyset$, let $a_\xi=\min (\calA^{h_\xi}\cap E)\backslash (B_\xi^*\cup\{\sup B_\xi^*\})\neq\emptyset$. Otherwise, let 
$a_\xi=\min\{a\in \calA^{h_\xi}\cap A: a>\sup B^*_\xi\}$. 
Let $$B_\xi=B_\xi^*\cup\{a_\xi\}\cup E\cap s_E(f(a_\xi))\cup\{s_E(f(a_\xi))\}.$$
Finally, take $A^*=\bigcup_{\xi<\kappa} B_\xi$. Then $A^*$ meets every boolean combination $\calA^h$ of $\calA$ on an unbounded set (witnessed by the $a_\xi$'s), $A^*\subseteq A$ and since $\{a_\xi\}_{\xi<\kappa}$ is unbounded in $\kappa$ and $$E\cap s_E(f(a_\xi))\subseteq B_\xi\subseteq A^*$$ for each $\xi$, we also have that $E\subseteq A^*$. Let $b<a$ be elements of $A^*$. If $a\in E$, then by definition of $C_f$ we have that $f(b)<a$. If $a\notin E$ and $a=a_{\xi+1}$ for some $\xi$, then 
$$a_\xi < f(a_\xi)< s_E(f(a_\xi))< a_{\xi+1}$$
by construction. If $\xi$ is a limit, $a=a_\xi$ and $f(a_\zeta)<s_E(f(a_\zeta))<a_\xi$ for each $\zeta<\xi$ again by construction. Since $b=a_\zeta$ for some $\zeta<\xi$, $f(b)<a$.
\end{proof}

The following Corollary will play an important role in the main result of the paper.

\begin{cor}\label{closure_club} Let $G$ be $\mathbb{P}_\calU$-generic, $f\in V\cap {^\kappa\kappa}$ be strictly increasing. Then there is $A\in \fil_{<\om, \kappa}(\calA_G)$ such that if $\{a(i)\}_{i\in\kappa}$ is the increasing enumeration of $A$ then $f(a(i))<a(i+1)$ for all $i \in \kappa$. 
\end{cor}
\begin{proof} Since $\fil_{<\om, \kappa}(\calA_G)$ is generated by $\calF_G$ (the set of second coordinated of elements of the generic filter $G$) we may use the previous lemma and get the result.	

Indeed, this is a standard density argument: Let $D_f$ be the set of all $(\calA,A)\in\mathbb{P}_\calU$ such that if $\{a(i)\}_{i<\kappa}$ is the enumeration function of $A$ then $f(a(i))< a(i+1)$ for all $i<\kappa$. By Lemma \ref{analogue_Qset} the set $D_f$ is dense. Thus, $G\cap D_f\neq\emptyset$ and so there is $(\calA,A)\in G\cap D_f$. But, by definition of the set $\calF_G$, we have that  $A\in \calF_G$, see also Remark \ref{generators_remark}, and by the same remark $A\in\fil_{<\omega,\kappa}(\calA_G)$. Since $(\calA, A)\in D_f$, we obtain that $f(a(i))<a(i+1)$.
\end{proof}

\section{Dense Maximality}\label{dense_maximality_section}

The notion of densely maximal independent families on $\omega$ appears for the first time in~\cite{MGSS}. Moreover, the maximal independent family constructed in~\cite{SS} which becomes a witness to $\mathfrak{i}=\aleph_1$ in the model of~\cite[Theorem 3.1]{SS} is densely maximal. A similar notion will play a vital role for our considerations:

\begin{defi}\label{def.dense.max} An independent family $\calA$ is said to be densely maximal if 
	for every $X\in [\kappa]^\kappa\backslash\calA$ and every $h\in\FF_{<\omega,\kappa}(\calA)$ there is $h'\in\FF_{<\omega,\kappa}(\calA)$ extending $h$ such that either 
	$\calA^{h'}\cap X=\emptyset$ or $\calA^{h'}\cap (\kappa\backslash X)=\emptyset$. 
\end{defi}

The following characterization of dense maximality on $\omega$ appears implicitly in the proof of~\cite[Theorem 3.1]{SS}. This characterization will be the main tool in showing that a specially designed normal measure supported $\kappa$-maximal independent family preserves its maximality after forcing with a large product of $\kappa$-Sacks forcing.

\begin{lem}\label{equiv}
Let $\calA$ be an independent family. Then $\calA$ is densely maximal if and only if

\medskip	
\noindent
$(\ast)$ $\forall h\in\FF_{<\omega,\kappa}(\calA) \forall X\subseteq\calA^h$ either there is $B\in\id_{<\omega,\kappa}(\calA)$ such that  $\calA^h\backslash X\subseteq B$, or there is $h'\in\FF_{<\omega,\kappa}(\calA)$ such that $h'\supseteq h$ and  $\calA^{h'}\subseteq \calA^h\backslash X$.	
\end{lem}
\begin{proof} Suppose $\calA$ satisfies property $(\ast)$. Let $X\in [\kappa]^\kappa$, $h\in \FF_{<\omega,\kappa}(\calA)$ and consider $Y=X\cap \calA^h$. Apply property $(*)$. If there is $B\in\id_{<\omega,\kappa}(\calA)$ such that $\calA^h\backslash X\subseteq B$, then $\calA^h\backslash X\in\id_{<\omega,\kappa}(\calA)$. Then there is $h'\supseteq h$ such that $\calA^{h'}\cap (\calA^h\backslash X)=\calA^{h'}\backslash X=\emptyset$.  If there is $h'\supseteq h$ such that $\calA^{h'}\subseteq \calA^h\backslash X$, then $\calA^{h'}\cap X=\emptyset$. Thus $\calA$ is densely maximal.
	
Now suppose $\calA$ is densely maximal. Fix $h\in\FF_{<\omega,\kappa}(\calA)$ such that $X\subseteq\calA^h$. We will show that $\calA$ satisfies property $(\ast)$. Suppose, there is no $B\in\id_{<\omega,\kappa}(\calA)$ such that $\calA^h\backslash X\subseteq B$. Thus in particular $\calA^h\backslash X\notin\id_{<\omega,\kappa}(\calA)$ and so there is $h'\in\FF_{<\omega,\kappa}(\calA)$ such that for all $h''\supseteq h'$
the set $\calA^{h''}\cap (\calA^h\backslash X)\neq\emptyset$. If $h$ and $h'$ are incompatible as conditions in $\FF_{<\omega,\kappa}(\calA)$, then 
$\calA^{h'}\cap (\calA^h\backslash X)=\emptyset$, which is a contradiction. 
Therefore $h$ and $h'$ are compatible. Without loss of generality, $h'\supseteq h$ (otherwise pass to a common extension of $h$ and $h'$).  Thus $h$ has an extension, namely $h'$, such that for all $h''\supseteq h'$ the set $\calA^{h''}\backslash X$ is non-empty. Apply the fact that $\calA$ is densely maximal to $\calA^{h'}$ and $X$.  Thus, there is
$h''\supseteq h'$ such that $\calA^{h''}\cap X=\emptyset$. 
Therefore $\calA^{h''}\subseteq\calA^{h'}\backslash X\subseteq\calA^h\backslash X$, which completes the proof of property $(\ast)$.
\end{proof}

\begin{lem}\label{ast_0}
	Let $G$ be $\bbP_\calU$-generic. Then in $V_0=V[G]$ the family 
	$\calA_G:=\bigcup\{\calA:\exists A(\calA,A)\in G\}$ is densely maximal.
\end{lem}
\begin{proof} It is sufficient to show that $\calA_G$ satisfies property $(\ast)$ from Lemma~\ref{equiv}. Thus, fix $h$ and $X$ as in $(\ast)$. 
Suppose there is no $B\in\id_{<\omega,\kappa}(\calA_G)$ such that $\calA_G^h\backslash X\subseteq B$. Then, in particular 	
$\calA_G^h\backslash X\notin\id_{<\omega,\kappa}(\calA_G)$ and so there is 
$h_0\in \FF_{<\omega,\kappa}(\calA_G)$ such that for all $h_1\supseteq h_0$ the set $\calA^{h_1}\cap (\calA^h\backslash X)\neq\emptyset$ (by definition of the density ideal). Consider the partition 
$$\calE=\{\calA^h\backslash X, \kappa\backslash (\calA^h\backslash X)\}$$
and the set $\calA^{h_0}$. By Corollary~\ref{density_partitions} there is $h_1\in\FF_{<\omega,\kappa}(\calA_G)$ extending $h_0$ such that 
$\calA^{h_1}$ is contained in one element of $\calE$. However, if 
$\calA^{h_1}\subseteq \kappa\backslash (\calA^h\backslash X)$, then 
$\calA^{h_1}\cap (\calA^h\backslash X)=\emptyset$, which is a contradiction to the choice of $h_0$. Thus $\calA^{h_1}\subseteq \calA^h\backslash X$ and so $\calA^{h_1}\cap X=\emptyset$.
\end{proof}

\section{Preprocessed Conditions and Outer hulls}\label{outer_hull}

In this section we introduce the notions of preprocessed conditions and outer hulls for the special case of $\kappa$-Sacks forcing and its products. Note that both of these notions play a key role in Shelah's proof of $\mathfrak{i}<\mathfrak{u}$ from~\cite{SS}: preprocessed conditions appear in~\cite[Claim 1.11]{SS}, while outer hulls appear in proof of~\cite[Theorem 3.1]{SS} (page 440 of the article). Throughout this section we work under the assumption of GCH (at least $2^\kappa=\kappa^+$ and $2^{<\kappa}=\kappa$) and $\kappa$ measurable. Thus in particular $\kappa$ is strongly inaccessible. We will work with the generalization of Sacks forcing and its products to the uncountable, both of which were first studied by Kanamori~\cite{AK}. 


Recall that $p \subseteq 2^{< \kappa}$ is a \emph{tree} if it is closed under initial segments. That is, $u \in p$ and $v \subseteq u$ imply $v \in p$. Whenever $p$ is a tree, $t,r\in p$ and $t$ is a proper initial segment or equal to $r$, we write $t\unlhd r$ and $r\unrhd t$. A node $u\in p$ \emph{splits} in $p$ if both $u^\frown 0$ and $u^\frown 1$ belong to $p$. Given a tree $p$, we denote by $\hbox{split}(p)$ the set of splitting nodes of $p$.

\begin{defi} 
For a strongly inaccessible $\kappa$, the $\kappa$-Sacks forcing, denoted $\mathbb{S}_\kappa$, is the poset consisting of sub-trees $p$ of $2^{<\kappa}$ such that:
\begin{enumerate}
\item for each $u \in p$ there is $t \in p$ such that  $u\unlhd t$ and $t$ splits in $p$ ($t$ is said to be a splitting extension of $u$);
\item for any $\alpha < \kappa$, if $(u_\beta: \beta < \alpha)$ is a sequence of nodes in $p$ such that $\beta < \gamma< \alpha \to u_\beta \subseteq u_\gamma$, then $\bigcup \{u_\beta: \beta < \alpha \} \in p$;
\item if $\delta < \kappa$ is a limit ordinal, $u \in 2^\delta$ and for arbitrarily large $\beta < \delta$ the node $u \restriction \beta$ splits in $p$, then $u$ splits in $p$.
\end{enumerate}
The extension relation on $\mathbb{S}_\kappa$ is defined by $p \leq q$ if and only if $p \subseteq q$. 
\end{defi}

As in the countable case we define the $\stem(p)$ where $p$ is a condition in $\mathbb{S}_\kappa$ as the unique splitting node that is comparable with all elements in $p$. By recursion on $\ka$ define: 

\begin{defi}[The $\alpha$-th \emph{splitting level} of $p$]
Given $p \in \mathbb{S}_\kappa$ let
\begin{itemize}
\item $\split_0 (p) = \stem(p)$,
\item $\split_{\alpha+1}(p) = \{ \stem(p_{u ^\frown i}) : u \in \split_\alpha(p)$ and $i \in 2 \}$,
\item for $\delta<\kappa$ is a limit ordinal, $\split_{\delta}(p)= \{s  \in p: s \hbox{ is a limit of nodes in }\bigcup_{\alpha< \delta} \split_\alpha(p)\} $.  
\end{itemize}
We refer to $\split_\alpha(p)$ as the $\alpha$-th splitting level of $p$.
Moreover for  $t\in\hbox{split}(p)$, let $\Sl(t,p)= \alpha$ where $t\in\hbox{split}_\alpha(p)$.
\end{defi} 

Using this splitting levels we define the \emph{fusion orderings} $\leq_\alpha$ on $\mathbb{S}_\kappa$: Given $q$ and $p$ in  $\mathbb{S}_\ka$ let $q \leq_\al p$ if and only if $q \leq p$ and $\split_\al(p)= \split_\al(q)$.

\begin{defi}
A \emph{fusion sequence} $(p_\al: \al< \ka) \subseteq \mathbb{S}_\kappa$ is sequence of conditions in $\bbS_\ka$ such that $p_{\alpha +1} \leq_\al p_\al$ for all $\al<\ka$ and whenever $\delta< \ka$ is a limit, then $p_\delta \leq_\al p_\al$ for all $\al< \delta$. 
\end{defi}

For any regular uncountable cardinal $\lambda$ we denote by $\mathbb{S}^\lambda_\kappa$ the $\kappa$-support product of $\lambda$ many copies of $\mathbb{S}_\kappa$. Moreover:

\begin{defi}[Product fusion, Definition 1.7 in \cite{AK}]\label{fusion} \hfill
\begin{itemize}
\item If $(p_\alpha: \alpha< \beta) \subseteq \bbS^\lambda_\kappa$, we define a condition $p = \bigwedge_{\alpha < \beta} p_\alpha$ with $\dom(p) = \bigcup_{\alpha< \beta} \dom(p_\alpha)$ and for every $\gamma \in \dom(p)$, $p(\gamma)= \bigcap \{ p_\alpha(\gamma): \gamma \in \dom(p_\alpha)\}$. 
Note that in the case $p \restriction \gamma \notin \bbS^\gamma_\kappa$ for $\gamma \in \dom(p)$ or $\lvert \dom(p) \lvert > \kappa$ then $p$ is left undefined.
\item If $p,q \in \bbS^\lambda_\kappa$, $\alpha< \kappa$ and $F \subseteq \dom(q)$ with $\lvert F \lvert \leq \kappa$, we say $p \leq_{F, \alpha} q$ if and only if $p \leq q$ and for every $\beta \in F$, $p(\be) \leq_\al q(\be)$.
\end{itemize}
\end{defi}

\begin{lem}[Generalized fusion \cite{AK}]\label{gfus}
Suppose $(p_\alpha : \alpha < \kappa)\subseteq \bbS^\lambda_\kappa$ and $ F_\al \subseteq \lambda$ have the following properties:
\begin{enumerate}
\item $p_{\alpha+1} \leq_{F_\alpha,\alpha} p_\alpha$ and $p_\delta = \bigwedge_{\alpha < \delta} p_\alpha$ when $\delta$ is a limit ordinal $< \!\kappa$.
\item $\lvert F_\al \lvert < \ka$, $F_\alpha \subseteq F_{\alpha+1}$, $F_\delta = \bigcup_{\alpha< \delta} F_\alpha$ for limit $\delta< \kappa$ and $\bigcup_{\alpha< \kappa} F_\alpha = \bigcup_{\alpha< \kappa} \dom(p_\alpha)$.
\end{enumerate}
Then $p = \bigwedge_{\alpha < \kappa} p_\alpha \in \bbS^\lambda_\kappa$ and we refer to $(p_\al, F_\al : \al < \ka)$ as a generalized fusion sequence. 
\end{lem}

In the following, we fix some notation:

\begin{defi} \hfill
\begin{itemize}
    \item Given a condition $p \in \bbS_\kappa^\lambda$, $\alpha< \kappa$ and $F \subseteq \supp(p)$ so that $\lvert F \lvert< \kappa$ let $\Lambda^F_\alpha(p)=\prod_{i\in F}\split_\alpha(p(i))$. That is 
    $$\Lambda^F_\alpha(p) = \{ \bar{\sigma} = (\sigma_i)_{i \in F} : \sigma_i \in \split_\al(p(i))\}.$$
    \item For all $\bar{\sigma} \in \Lambda^F_\alpha(p)$ let  $p_{\bar{\sigma}}\leq p$ be defined as follows: $\supp(p)=\supp(p_{\bar{\sigma}})$ and

\[p_{\bar{\sigma}}(i) =
\begin{cases}
(p(i))_{ \sigma_i} & \text{if } i \in F \\
p(i)  & \text{ otherwise }
\end{cases}
\] 
    \item Given $h \in {^F2}$ and $\bar{\sigma} \in \Lambda^F_\alpha(p)$, let $p^h_{\bar{\sigma}}$ be defined as follows:  $\supp(p^h_{\bar{\sigma}})=\supp(p)$ and 

\[p^h_{\bar{\sigma}}(i) =
\begin{cases}
(p(i))_{ \sigma_i^\smallfrown h(i)} & \text{if } i \in F \\
p(i)  & \text{ otherwise }
\end{cases}
\] 
\end{itemize}
\end{defi}

\subsection{Preprocessed conditions}

The following notion of begin preprocessed for a condition can be seen for example in~\cite[Lemma 1.11]{SS}. Given a name for a real $\tau$  and a condition $p$ in an appropriate partial order, in particular the poset $Q_\mathcal{I}$ from~\cite{SS}, Sacks forcing, or Miller partition forcing (see for definition for example~\cite{RestMad}), the notion provides a sufficiently good ground model approximation of $\tau$  realized by a condition $q$ stronger than $p$. 

We adapt the notion in the context of $\kappa$-Sacks forcing and  its products.

\begin{defi}  \hfill
	\begin{enumerate}
		\item Let $\dot{X}$ be a $\mathbb{S}_\kappa$-name for a subset of $\kappa$. We say that $p\in\mathbb{S}_\kappa$ is preprocessed for $\dot{X}$ if for all $\alpha\in\kappa$ and all $t\in\split_\alpha(p)$ there is $x_t\in{^\alpha 2}$ such that $p_t\Vdash\chi_{\dot{X}}\upharpoonright\alpha=\check{x}_t$.
		\item 
		
		\begin{enumerate}
		    \item Let $\dot{X}$ be a $\mathbb{S}_\kappa^\lambda$ name for a subset of $\kappa$, $p \in \mathbb{S}_\kappa^\lambda$ and $F \subseteq \supp(p)$ with $\lvert F \lvert < \kappa$. We say that $p\in\mathbb{S}_\kappa^\lambda$ is preprocessed for the pair $(F,\dot{X})$ if for all $\alpha<\kappa$ and all $\bar{\sigma}\in\Lambda_\alpha^F(p)$ there are $F' \supseteq F$, such that $F' \subseteq \dom(p)$ and $\lvert F'\lvert < \kappa$, $\bar{\tau}_{\bar{\sigma}} \in \Lambda_\alpha^{F'}(p)$ and $x_{\bar{\sigma}}\in {^\alpha 2}$ such that $\bar{\sigma} \sqsubseteq \bar{\tau}_{\bar{\sigma}}$ and $p_{\bar{\tau}_{\bar{\sigma}}}\Vdash\chi_{\dot{X}}\upharpoonright \alpha =\check{x}_{\bar{\sigma}}$.\footnote{Here $\bar{\sigma} \sqsubseteq \bar{\tau}_{\bar{\sigma}}$ means for all $i \in F$ $\sigma_i= \tau_i$. }
		    
		    \item We say that a condition $p \in \mathbb{S}_\kappa^\lambda$ is preprocessed for the name $\dot{X}$ if for all $F \subseteq \dom(p)$ such that $\lvert F \lvert < \kappa$, $p$ is preprocessed for $(\dot{X}, F)$.
		\end{enumerate}

	\end{enumerate}
\end{defi}

\begin{remark} $ $
	\begin{enumerate}
		\item Note that if $p\in\mathbb{S}_\kappa$ is preprocessed for $\dot{X}$, then for each $\alpha\in \kappa$ there is $Y_\alpha\subseteq {^\alpha 2}$ such that $p\Vdash \chi_{\dot{X}}\upharpoonright \alpha\in\check{Y}_\alpha$. Indeed, take $Y_\alpha =\bigcup\{x_t:t\in \split_\alpha(p)\}$ where $x_t$ is defined as in the definition above.
		\item Similarly, if $p\in\mathbb{S}_\kappa^\lambda$ is preprocessed for $(F,\dot{X})$ ($F$ like above), then for each $\alpha<\kappa$ there is $Y_\alpha\subseteq {^\alpha 2}$ such that  $p\Vdash \chi_{\dot{X}}\upharpoonright \alpha\in\check{Y}_\alpha$. Just take $Y_\alpha = \bigcup\{ x_{\bar{\sigma}}:\bar{\sigma}\in \Lambda^{F}_\alpha(p) \}$ where $x_{\bar{\sigma}}$ is defined as above.
    \end{enumerate}
\end{remark}

\begin{lem}\label{preprocessed} \hfill
\begin{enumerate}
    \item Let $p\in\mathbb{S}_\kappa$ and let $\dot{X}$ be a $\mathbb{S}_\kappa$-name for a subset of $\kappa$. Then there is $q\leq p$ such that $q$ is preprocessed for $\dot{X}$.
    \item Let $p\in\mathbb{S}_\kappa^\lambda$ and let $\dot{X}$ be a $\mathbb{S}^\lambda_\kappa$-name for a subset of $\kappa$. Then, for all $F \subseteq \supp(p)$ with $\lvert F\lvert< \kappa$ and $\gamma<\kappa$ there is $q\leq_{F, \gamma} p$ such that $q$ is preprocessed for $(F,\dot{X})$.
\end{enumerate}
\end{lem}
\begin{proof}
$(1)$ We build a fusion sequence $\langle q_\alpha:\alpha<\kappa\rangle$ 
below $p$ such that for all $\alpha>0$ and all $t\in\split_\alpha(q_\alpha)$ there is $x_t\in{^\alpha 2}$ such that 
$(q_\alpha)_t\Vdash \chi_{\dot{X}}\upharpoonright \alpha =\check{x}_t$. 
Start with $q_0=p$. Consider $t\in\split_0(p)$, i.e. $t=\stem(p)$. For each $i\in\{0,1\}$ there is $w_{t,i}\leq p_{t^\smallfrown i}$ and $x(t,i)\in\{0,1\}$ such that 
$w_{t,i}\Vdash\chi_{\dot{X}}(0)=\check{x}(t,i)$. 
Note that $\hbox{stem}(w_{t,i})\unrhd \stem(p_{t^\smallfrown i})\unrhd t^\smallfrown i$. Define $q_1= w^0_{t,0}\cup w^0_{t,1}$. Then $q_1\leq_0 q_0$ and for all $s\in\split_1(q_1)$ there is $x_s\in {^12}$ such that $(q_1)_s\Vdash\chi_{\dot{X}}\upharpoonright 1=\check{x}_s$. Indeed, if $s\in\split_1(q_1)$ then $s\unrhd t^\smallfrown i$ for $i\in\{0,1\}$ and so $(q_1)_s=w_{t,i}$. Thus $(q_1)_s\Vdash\chi_{\dot{X}}\upharpoonright 1= x_s$ where $x_s=(0,x(s,i))$.

Now, suppose $q_\alpha$ has been defined and $\forall t\in\split_\alpha(q_\alpha)$ there is $x_t\in{^\alpha 2}$ such that $(q_\alpha)_t\Vdash\chi_{\dot{X}}\upharpoonright \alpha=\check{x}_t$. For each $t\in\split_\alpha(q_\alpha)$ and each $i\in\{0,1\}$ find $w_{t,i}\leq (q_\alpha)_{t^\smallfrown i}$ and $x(t,i)\in\{0,1\}$ such that $w^\alpha_{t,i}\Vdash\chi_{\dot{X}}(\alpha)=\check{x}(t,i)$. Then,
take $q_{\alpha+1}=\bigcup\{w_{t,i}: t\in\split_\alpha(q_\alpha), i\in\{0,1\}\}$. Then $\split_\alpha(q_{\alpha+1})=\split_\alpha(q_\alpha)$ and so $q_{\alpha+1}\leq_\alpha q_\alpha$. Moreover, for all $t\in\split_{\alpha+1}(q_{\alpha+1})$ there is $x_t\in{^{\alpha+1} 2}$ such that 
$(q_{\alpha+1})_t\Vdash\chi_{\dot{X}}\upharpoonright \alpha+1 =\check{x}_t$. Indeed. Fix $t\in\split_{\alpha+1}(q_{\alpha+1})$. Thus $r^\smallfrown i\unlhd t$ for some $r\in \split_\alpha(q_{\alpha+1})=\split_\alpha(q_\alpha)$ for some $i\in\{0,1\}$. By Inductive Hypothesis $(q_\alpha)_r\Vdash\chi_{\dot{X}}\upharpoonright \alpha=\check{x}_r$ for some $x_r\in{^\alpha 2}$. However $t\unrhd r^\smallfrown i$ and so $(q_{\alpha+1})_t=w_{r,i}\leq (q_\alpha)_{r^\smallfrown i}\leq (q_\alpha)_r$. Thus, $(q_{\alpha+1})_t\Vdash\chi_{\dot{X}}\upharpoonright \alpha=\check{x}_r\hbox{ and }\chi_{\dot{X}}(\alpha)=\check{x}(r,i)$. That is $(q_{\alpha+1})_t\Vdash\chi_{\dot{X}}\upharpoonright \alpha+1 =\check{x}_t$ where $x_t=x_r\cup\{(\alpha, x(r,i))\}$. 

It remains to consider the limit case. Suppose $\langle q_\beta:\beta<\alpha\rangle$ have been defined and for all $\beta<\alpha$ and all $t\in\split_\beta(q_\beta)$ there is $x_t\in{^\beta 2}$ such that $(q_\beta)_t\Vdash \chi_{\dot{X}}\upharpoonright \beta=\check{x}_t$. Then 
take $q_\alpha=\land_{\beta<\alpha} q_\beta$. Note that if $t\in\split_\alpha(q_\alpha)$ then there is $\{\eta_\xi:\xi<\alpha\}$ unbounded in $\alpha$ such that $t\upharpoonright \eta_\xi\in\split_\xi(q_\xi)$ and so by inductive hypothesis for some $x_{t\upharpoonright \eta_\xi}\in {^\xi 2}$ we have $(q_\xi)_{t\upharpoonright \eta_\xi}\Vdash \chi_{\dot{X}}\upharpoonright \xi =\check{x}_{t\upharpoonright \eta_\xi}$. Then for $x_t=\bigcup\{x_{t\upharpoonright \eta_\xi}:\xi <\alpha\}$ we have $(q_\alpha)_t\Vdash\chi_{\dot{X}}\upharpoonright \alpha=\check{x}_t$. 

$(2)$ The argument for the product runs similarly as the above case: Let $p \in \mathbb{S}_\kappa^\lambda$, $\gamma < \kappa$ and $F \subseteq \supp(p)$ so that $\lvert F \lvert< \kappa$. We shall define a fusion sequence $\langle q_\alpha, F_\alpha: \alpha<\kappa\rangle\subseteq \mathbb{S}_\kappa^\lambda$ below $p$, ordinals $(\eta_\alpha: \alpha< \kappa)$ and bijections $g_\alpha: F_\alpha \to \eta_\alpha$ such that for all $\alpha<\kappa$:

\begin{enumerate}
    \item $q_{\alpha+1} \leq_{F_\alpha, \gamma+ \alpha} q_\alpha$,
    \item $\eta_\alpha \geq \alpha$,
    \item For all $\alpha< \alpha'<\kappa$, $g_\alpha \subseteq g_{\alpha'}$ and for limit ordinals $\delta< \kappa$ $g_\delta = \bigcup_{\alpha< \delta} g_\alpha$,
    \item For all ${\bar{\sigma}}\in\Lambda_{\gamma + \alpha}^{F_\alpha}(q_\alpha) $ there is $x_{\bar{\sigma}}\in{^\alpha 2}$ such that $(q_{\alpha+1})_{\bar{\sigma}}\Vdash\chi_{\dot{X}}\upharpoonright\alpha=\check{x}_{\bar{\sigma}}$.
\end{enumerate}

Since similar arguments will be used in the upcoming results we give the proof in full detail. Start with $q_0 = p$ and $F_0= F$, in order to arrange that $q_{\alpha+1} \leq_{F_\alpha, \gamma+ \alpha} q_\alpha$ start the following construction at some indecomposable ordinal $\alpha >\gamma$ (otherwise for all ordinals $\beta < \alpha$, $\beta+\alpha = \alpha$ and so, at $\alpha$ we would just get $q_{\alpha+1} \leq_{F_\alpha,\alpha} q_\alpha$) and letting $q_\beta = p$ for all $\beta<\alpha$. At limit stages $\delta<\kappa$ the construction of $q_\delta$ and $F_\delta$ is  determined by the conditions in Definition~\ref{fusion}. Let $\eta_\delta = \sup_{\alpha<\delta} \eta_\alpha$ and $g_\delta$ as above.
 
Finally for the successor case, suppose $q_\alpha$, $\eta_\alpha$, $g_\alpha$ and $F_\alpha$ have been defined. Fix an enumeration 
$\{\gamma_l=(\bar{\sigma}_l, h_l): l< \rho\}$ of all pairs of the form $(\bar{\sigma}, h)$ such that $\bar{\sigma} \in \Lambda_\sigma^{F_\alpha}(q_\alpha)$ and $h \in {^{F_\alpha}2}$. Note that the ordinals $\rho$ is $<\kappa$.

Inductively, we will construct a sequence $\{r^\alpha_l:l<\rho\}$ of conditions below $q_\alpha$ satisfying:

\begin{enumerate}
    \item $r^\alpha_0 = q_\alpha$,
    \item $r^\alpha_{l+1} \leq_{F_\alpha,\gamma +\alpha} r^\alpha_{l}$,
    \item $(r^\alpha_{l+1})_{\bar{\sigma}_l}^{h_l}$ forces a value $\check{x}(\alpha,l)$ for $\chi_{\dot{X}}\upharpoonright \alpha$,
    \item For $l$ limit ordinal $r^\alpha_l = \bigwedge_{k< l} r^\alpha_k$.
\end{enumerate}

It is enough to explain how the successor step is built: Suppose that we have constructed $r^\alpha_l$ satisfying the conditions above and consider the pair $\gamma_l =(\bar{\sigma}, h)$, then find a condition $w_l \leq (r_l)_{\bar{\sigma}}^{h}$ forcing a value $x(\bar{\sigma},l)$ for $\chi_{\dot{X}}\upharpoonright \alpha$. Note that $w_l$ is clearly not a condition that satisfies (2), so we build $r^\alpha_{l+1}$ as follows: $\supp(r^\alpha_{l+1})=\supp(w_l)$ and 

\[r^\alpha_{l+1}(i) =
\begin{cases}
(w_l(i)) \cup \{ (q_\alpha(i))_{\tau ^\frown j}: \tau \in \split_\alpha (r_l(i)) \setminus \sigma_i \hbox{ or } j= 1-h(i)  \} & \text{if } i \in F_\alpha  \\
 w_l(i) & \text{ otherwise }
\end{cases}
\]
\medskip

Note that this is now a condition satisfying the properties above. Put $q_{\alpha+1}= \bigwedge_{l< \rho} r^\alpha_l$ $F_{\alpha+1}=F_\alpha\cup\{\min(\supp(q_{\alpha+1})\backslash F_\alpha)\}$, $\eta_{\alpha+1} = \eta_\alpha +1$ and $g_{\alpha+1} = g_\alpha \cup \{(\min(\supp(q_{\alpha+1})\backslash F_\alpha), \eta_\alpha)\}$. 
\medskip

In order to finish the proof, we check that the condition $q$ is indeed preprocessed for $(F, \dot{X})$. Fix $\alpha< \kappa$ and $\bar{\sigma} \in \Lambda_{\alpha}^{F}(q)$. First, extend $\bar{\sigma}$ to a sequence $\bar{\sigma}'= \bar{\sigma}^\frown \bar{\rho}_0$ in $\Lambda_{\alpha}^{F_\alpha}(q)$ by adding a fix tail $\bar{\rho}_0$ of splitting nodes in $q(i)$ at level $\alpha$ for each $i \in F_\alpha \backslash F$. 
 
Since $q \leq_{\alpha, F_\alpha} q_\alpha$ we get that $\Lambda_{\alpha}^{F_\alpha}(q)= \Lambda_{\alpha}^{F_\alpha}(q_{\alpha})$ and so $\bar{\sigma}' \in \Lambda_{\alpha}^{F_\alpha}(q_{\alpha})$ and we can use the fusion properties to get that there is a $x_{\bar{\sigma}'} \in {^\alpha 2}$ such that $(q_{\alpha+1})_{\bar{\sigma}'} \Vdash \chi_{\dot{X}}\upharpoonright (\alpha) = \check{x}_{\bar{\sigma}'}$, so we finally use that the condition $q_{\bar{\sigma}'} \leq (q_{\alpha+1})_{\bar{\sigma}'}$ to get that $(q)_{\bar{\sigma}'} \Vdash \chi_{\dot{X}}\upharpoonright (\alpha) = \check{x}_{\bar{\sigma}'}$ as we wanted.

\end{proof}

\begin{cor}
    Let $p\in\mathbb{S}_\kappa^\lambda$ and let $\dot{X}$ be a $\mathbb{S}^\lambda_\kappa$-name for a subset of $\kappa$, then we can find $q \leq p$ such that $q$ is preprocessed for $\dot{X}$.
\end{cor}

\begin{proof}
 Recall first that there are $\kappa$-many sets $G \subseteq \dom(p)$ so that $\lvert G \lvert < \kappa$, so let $(G_\beta: \beta <\kappa)$ be an enumeration of them. We can use build a fusion sequence $(q_\alpha, F_\alpha)$ below $p$ such that for all $\alpha< \kappa$, $q_{\alpha + 1}$ is preprocessed for $(\dot{X}, G_\alpha)$.
 
 Start with $q_0 = p$ and $F_0 = G_0$. The limit step is built as usual and it is left to explain how the successor case is constructed: Suppose we have already defined $q_\alpha$ and $F_\alpha$, then we use the Lemma above for $\dot{X}$, $q_\alpha$ and $F_\alpha$ to get a condition $r \leq_{\alpha, F_\alpha} q_\alpha$ that is preprocessed for $(\dot{X}, G_\alpha)$. Define then $q_{\alpha+1} =r$ and $F_{\alpha+1} = F_\alpha \cup G_\alpha$.
 
 We claim that the fusion $q$ of the sequence defined above is preprocessed for all $G_\alpha$'s: Indeed, if $\alpha$ is fixed we know that $q_{\alpha+1}$ is preprocessed for $(\dot{X}, G_\alpha)$, then it is enough to notice that since $q \leq_{\alpha+1, F_{\alpha+1}} q_{\alpha+1}$, then $q$ is also preprocessed for $(\dot{X}, G_\alpha)$.
\end{proof}

\subsection{Outer hull}

Below, we introduce the notion of an outer hull, making explicit some well-known techniques. In the consistency proof of $\mathfrak{i}<\mathfrak{u}$ from~\cite{SS}, outer hulls appear on page 440. Other more recent applications of the notion can be found for example  in~\cite{RestMad}.

\begin{defi}\hfill
\begin{itemize}
    \item Let $p\in \mathbb{S}_\kappa$ and let $\dot{X}$ be a $\mathbb{S}_\kappa$-name for a subset of $\kappa$. For each $t\in\hbox{split}_\alpha(p)$, we refer to the set 
$Y_t=\{\beta\in\kappa: p_t\not\Vdash\check{\beta}\notin\dot{X}\}$, as the 
outer hull of $\dot{X}$ below $p_t$. Moreover, if $q_{t,\beta}\leq p_t$ and $q_{t,\beta}\Vdash\check{\beta}\in\dot{X}$,  we say that $q_{t,\beta}$ is a witness for $\beta\in Y_t$.
    \item Let $p\in \mathbb{S}^\lambda_\kappa$ and let $\dot{X}$ be a $\mathbb{S}^\lambda_\kappa$-name for a subset of $\kappa$. For all $\alpha< \kappa$, $F \subseteq \supp(p)$ such that $\lvert F \lvert < \kappa$ and $\bar{\sigma} \in \Lambda^F_\alpha(p)$ we refer to the set 
$Y_{\bar{\sigma}}=\{\beta\in\kappa: p_{\bar{\sigma}}\not\Vdash\check{\beta}\notin\dot{X}\}$, as the 
outer hull of $\dot{X}$ below $p_{\bar{\sigma}}$. Moreover, if $q_{\bar{\sigma},\beta}\leq p_{\bar{\sigma}}$ and $q_{\bar{\sigma},\beta}\Vdash\check{\beta}\in\dot{X}$,  we say that $q_{\bar{\sigma},\beta}$ is a witness for $\beta \in Y_{\bar{\sigma}}$.
\end{itemize}
\end{defi}

\begin{remark} Let $\dot{X}$ be a $\mathbb{S}_\kappa$-name for a subset of $\kappa$. Suppose $Y_t$ is the outer hull of $\dot{X}$ below $p_t$. If $p_t\Vdash\check{\beta}\in\dot{X}$, then $p_t$ is a witness to $\beta\in Y_t$
and so $p_t\Vdash \dot{X}\subseteq \check{Y}_t$. Analogously for $\mathbb{S}^\lambda_\kappa$ and  $Y_{\bar{\sigma}}$.
\end{remark}

We proceed with the following lemma.

\begin{lem} \hfill
\begin{enumerate}
    \item Let $p\in\mathbb{S}_\kappa$ be preprocessed for $\dot{X}$ and for each $t\in\split(p)$ let $Y_t$ be the outer hull of $\dot{X}$ below $p_t$. Then for each $\alpha< \kappa$, $t\in\hbox{split}_\alpha(p)$ and $\beta\in Y_t$ there is a $r(t,\beta)\in p$ such that $t$ is an initial segment of $r(t,\beta)$ and $p_{r(t,\beta)}\Vdash\check{\beta}\in\dot{X}$. Moreover, for each $t\in\hbox{split}(p)$ and each $\beta\in  Y_t$ there is $r=r(t,\beta)\in \hbox{split}_\beta(p)$ such that $p_r\Vdash \check{\beta}\in\dot{X}$.
    \item Let $p\in\mathbb{S}^\lambda_\kappa$ be preprocessed for $(\dot{X}, F)$ (where $F \subseteq \supp(p)$ and $\lvert F\lvert< \kappa$) and for each $\bar{\sigma} \in \Lambda^F_\alpha(p)$ let $Y_{\bar{\sigma}}$ be the outer hull of $\dot{X}$ below $p_{\bar{\sigma}}$. Then for each $\alpha< \kappa$, $\bar{\sigma} \in \Lambda^F_\alpha(p)$, $\beta\in Y_{\bar{\sigma}}$ and $i \in F$ there are $F' \supseteq F$, $\lvert F' \lvert<\kappa$ and $\bar{\tau}' \in \Lambda_\alpha^{F'}(p)$ such that if $\bar{\tau}'=(r_i(\bar{\sigma}, \beta) : i \in F')$ then $r_i(\bar{\sigma},\beta)\in p(i)$ such that $\sigma_i$ is an initial segment of $r_i(\bar{\sigma},\beta)$ for all $i \in F$ and $p_{\bar{\tau}'}\Vdash\check{\beta}\in\dot{X}$. Moreover, for each $\bar{\sigma} \in \Lambda^F_\alpha(p)$, $\beta\in  Y_{\bar{\sigma}}$ and $i \in F$ there are $F' \supseteq F$, $\lvert F' \lvert<\kappa$, $\bar{\tau}' \in \Lambda_\alpha^{F'}(p)$  such that if $\bar{\tau}=(r_i(\bar{\sigma}, \beta) : i \in F')$ where $r_i=r_i(\bar{\sigma},\beta)\in \hbox{split}_\beta(p)(i)$ for all $i\in F$, then $p_{\bar{\tau}'}\Vdash \check{\beta}\in\dot{X}$.
\end{enumerate}
\end{lem}
\begin{proof}
$(1)$ Fix $\alpha,t$ and $\beta\in Y_t$. Then there is $q \leq p_{t}$ such that $q\Vdash \beta \in\dot{X}$.  Take any $r= r(t, \beta)\in \hbox{split}_\beta(q)$. Note that $t\unlhd r$ and there is $\beta'\geq \beta$ such that $r\in \hbox{split}_{\beta'}(p)$. On the other hand, because $p$ is preprocessed we can use Lemma~\ref{preprocessed} to find $x_r \subseteq {^{\alpha'} 2}$ such that
$p_r\Vdash \chi_{\dot{X}} \upharpoonright \beta' = \check{x}_r$. Since $q_r\leq p_r$ we must have that $ \check{x}_t(\beta)=1$. Thus $p_r\Vdash \beta\in \dot{X}$.
	
Moreover, if $r^*\unlhd r$ and $r^*\in\hbox{split}_{\beta+1}(p)$ then already $p_{r^*}\Vdash \check{\beta}\in\dot{X}$. Indeed, since $p$ is preprocessed there is an $x_{r^*} \in {^{\beta} 2}$ so that $p_{r^*} \Vdash \chi_{\dot{X}} \upharpoonright_{\beta+1} \in \check{x}_{r^*}$, thus $p_r$ forces a value for $\chi_{\dot{X}}(\beta)$ which has to be one because $q_r \leq q,p_r$ and so $p_r \Vdash \check{\beta} \in \dot{X}$. In particular, if $\beta\leq \Sl(t,p)$ then in fact $p_t\Vdash \check{\beta}\in \dot{X}$.

\medskip
\noindent
$(2)$ In the same way, fix now $\alpha, F$ and $\bar{\sigma} \in \Lambda_\alpha^F(p)$. Then there is $q \leq p_{\bar{\sigma}}$ such that $q\Vdash \beta \in\dot{X}$. For all $i \in F$ take any $r_i= r_i(\bar{\sigma}, \beta)\in \hbox{split}_\beta(q(i))$. Note that $\sigma_i\unlhd r_i$ for all $i \in F$. We can find $\beta'\geq \beta$ such that $r_i \in \hbox{split}_{\beta'}(p(i))$ uniformly for all $i \in F$. Hence if $\bar{\tau}=(r_i(\bar{\sigma}, \beta) : i \in F)$, by Lemma~\ref{preprocessed} there are $F' \supseteq F$, $\bar{\tau}' \in \Lambda_\alpha^{F'}(p)$ and $x_{\bar{\tau}'} \in {^{\beta'} 2}$ such that $p_{\bar{\tau}'}\Vdash \chi_{\dot{X}} \upharpoonright \beta' = \check{x}_{\bar{\tau}}$. Let $y= \chi_{\dot{X}} \upharpoonright \beta'$. Since $q_{\bar{\tau}'}\leq p_{\bar{\tau}'}$ we must have that $ y(\beta)=1$. Thus $p_{\bar{\tau}'}\Vdash \beta\in \dot{X}$. The argument for the moreover part is the same as for part (1). 
\end{proof}
	
\begin{defi} For $p\in\mathbb{S}_\kappa^\lambda$ and  $F\in[\hbox{supp}(p)]^{<\kappa}$, let $\Lambda(F)=\bigcup_{\alpha<\kappa}\Lambda^F_\alpha$. For	$\bar{\sigma}$ and $\bar{\tau}$ in $\Lambda(F)$ we say that $\bar{\tau}\unlhd\bar{\sigma}$ if and only if for each $i\in F(\bar{\tau}(i)\unlhd\bar{\sigma}(i))$.
\end{defi}

\begin{cor}\label{preprocessed_hull} $ $
\begin{enumerate}
\item Let $\dot{X}$ be a $\mathbb{S}_\kappa$-name for an infinite subset of $\kappa$. If $p\in \mathbb{S}_\kappa$ is preprocessed for $\dot{X}$, $t\in\split(p)$ and $Y_t$ is the outer hull of $\dot{X}$ below $p_t$ then	$$Y_t=\{\beta< \kappa:\exists r\in \hbox{split}(p)\hbox{ such that }
t\unlhd r
\hbox{ and }p_r\Vdash\check{\beta}\in\dot{X}\}.$$	
\item Let $\dot{X}$ be a $\mathbb{S}_\kappa^\lambda$-name for an infinite subset of $\kappa$. If $p\in\mathbb{S}_\kappa^\lambda$ is preprocessed for $(\dot{X}, F\in[\hbox{supp}(p)]^{<\kappa})$ and $\bar{\sigma}\in\Lambda(F)$, then
$$Y_{\bar{\sigma}}=\{\beta\in\kappa:\exists \bar{\tau}\in \Lambda(F') \hbox{ for } F' \supseteq F, \lvert F'\lvert <\kappa \hbox{ such that  } \bar{\sigma}\sqsubseteq \bar{\tau}
\hbox{ and }p_{\bar{\tau}}\Vdash\check{\beta}\in\dot{X}\}.$$
\end{enumerate}
\end{cor}

\section{$\kappa$-Sacks indestructibility}\label{Sacks_indestructibility_section}

In this section we set out to obtain our main result, namely the relative consistency of $\mathfrak{i}(\kappa)<2^\kappa$. For this we start with a measurable cardinal $\kappa$ and a normal measure $\calU$ on $\kappa$. Using the forcing notion $\mathbb{P}_\calU$ we adjoin a $\calU$-supported $\kappa$-maximal independent family, which by Lemma~\ref{ast_0} is densely maximal. A key feature of our proof is Lemma~\ref{equiv} which gives an equivalent characterisation of dense maximality via the property $(\ast)$ of the same Lemma. In fact, to show that $\calA=\calA_G$ for a $\mathbb{P}_\calU$-generic filter $G$ remains maximal after forcing with a large product of $\kappa$-Sacks forcing, we show that the property $(\ast)$ is preserved. With other words, we show that in the final generic extension a $\kappa$-independent family which we adjoin via forcing at an initial step of the construction satisfies property $(\ast)$. Note that in difference with the original work of Shelah~\cite{SS}, we are only interested in products of $\kappa$-Sacks forcing and not iterations.

\begin{thm}\label{preservation.thm} (GCH) Let $\kappa$ be a measurable cardinal, $\calU$ a normal measure on $\kappa$ and let $G$ be $\mathbb{P}_\calU$-generic filter over the ground model $V_0$. Let $\calA=\calA_G$ and $V=V_0[G]$. Then $$V^{\bbS_{\kappa}}\vDash\calA\hbox{ is a densely maximal independent family}.$$
\end{thm}
\begin{proof} Note that GCH holds in $V$ and $\kappa$ is inaccessible in $V$. By Lemma~\ref{ast_0} the family $\calA$ is densely maximal in $V$. To prove that $(\calA\hbox{ is densely maximal})^{V^{\mathbb{S}_\kappa}}$ we will show that in $V^{\mathbb{S}_\kappa}$, property $(*)$ of $\calA$ from Lemma~\ref{equiv} holds. More precisely, we will show that in $V^{\mathbb{S}_\kappa}$ for each $X\subseteq\kappa$ and each $h\in\FF_{<\omega,\kappa}(\calA)$ such that $X\subseteq\calA^h$ property $(*)_{X,h}$ holds, where
			
\medskip
\noindent
$(*)_{X,h}$ either $\exists B\in\id_{<\omega,\kappa}(\calA)$ such that $\calA^h\backslash X\subseteq B$ or there is $h'\supseteq h$ such that $\calA^{h'}\subseteq\calA^h\backslash X$.
		
\medskip		
\noindent			
Suppose not. Thus, there are $X\subseteq\kappa$, $h\in\FF_{<\omega,\kappa}(\calA)$ such that $X\subseteq \calA^h$ and $\lnot(*)_{X,h}$. That is,
$$V^{\mathbb{S}_\kappa}\vDash X\subseteq\calA^h\land \calA^h\backslash X\notin\id_{<\omega,\kappa}(\calA)\land\forall h'\supseteq h(\calA^{h'}\cap X\neq\emptyset).$$
Let $\dot{X}$ be a $\mathbb{S}_\kappa$-name for $X$ in $V$ and let $p\in\mathbb{S}_\kappa$ force the above. By Lemma~\ref{preprocessed} we can assume that $p$ is preprocessed for $\dot{X}$ and by Corollary~\ref{preprocessed_hull} that for each $t\in\split(p)$ for the outer hull $Y_t$ of $\dot{X}$ below $p_t$ is of the form
$Y_t=\{\beta< \kappa:\exists r\in \hbox{split}(p)\hbox{ such that }t\unlhd r\hbox{ or }r\unlhd t\hbox{ and }p_r\Vdash\check{\beta}\in\dot{X}\}$.

\begin{clm}\label{Claim43} Let $t\in\split(p)$. Then $Y_t\subseteq \calA^h$.
\end{clm}			
\begin{proof}
	Let $m\in Y_t$. Thus there is $q_{t,m}\leq p_t$ such that $q_{t,m}\Vdash \check{m}\in\dot{X}$. But $p_t\Vdash\dot{X}\subseteq \calA^h$ and so $m$ must be an element of $\calA^h$.
\end{proof}

We will make use of the following function $H\in {^\kappa\kappa}\cap V$. Given $t\in\split(p)$ and $\beta\in Y_t$, let $r(t,\beta)$ be a witness to $\beta\in Y_t$ such that $t$ is comparable with $r(t,\beta)$ and $r(t,\beta)$ is of least splitting level. Then define 
$$H(\gamma)=\sup\{\gamma+1\}\cup\{\Sl(r(t,\beta)): t\in\hbox{split}_\gamma(p), \beta\leq \gamma \},$$
where if $r(t,\beta)$ is not defined, i.e. $\beta \notin Y_t$, then we take $\Sl(r(t,\beta))=0$.

Fix $t\in\split(p)$. Since $Y_t\subseteq \calA^h$, by the dense maximality of $\calA$ in $V$ either there is $B\in\id_{<\omega,\kappa}(\calA)$ such that $\calA^h\backslash Y_t\subseteq B$ or there is $h'\supseteq h$ such that $\calA^{h'}\subseteq \calA^h\backslash Y_t$. In the latter case, $\calA^{h'}\cap Y_t=\emptyset$ and since $p\Vdash\dot{X}\subseteq\check{Y}_t$, we obtain that $p\Vdash\calA^{h'}\cap\dot{X}=\emptyset$, contrary to the choice of $p$. Thus, we can assume that $\forall t\in\split(p)\exists B_t\in\id_{<\omega,\kappa}(\calA)$ such that $\calA^h\backslash Y_t\subseteq B_t$, $\calA^h\backslash Y_t\in\id_{<\omega,\kappa}(\calA)$. But then $Y_t\cup\kappa\backslash \calA^h\in\fil_{<\omega,\kappa}(\calA)$. Now, since $\fil_{<\omega,\kappa}(\calA)$ is a $\kappa$-p-set, there is $C\in\fil_{<\omega,\kappa}(\calA)$ such that $C\subseteq^* Y_t\cup\kappa\backslash\calA^h$ for each $t\in\split(p)$. Thus in particular, $C\cap\calA^h\subseteq^* Y_t$ for all $t\in\split(p)$ and so we can find a function $f\in V\cap{^\ka\ka}$ such that
$$\forall \alpha\in\kappa \,\,\forall t\in\hbox{split}_{\alpha+1}(p) \,\,(C\cap\calA^h)\backslash Y_t\subseteq f(\alpha).$$
Equivalently, for each $\alpha\in\kappa$, 
$$(C\cap\calA^h)\backslash f(\alpha)\subseteq \bigcap_{t\in\split_{\alpha+1}(p)}Y_t.$$
Moreover, we can assume that $H(\alpha)+1 < f(\alpha)$ and that $f$ is strictly increasing.

Now, we use Corollary~\ref{closure_club} for the strictly increasing function $f^2=f\circ f$, i.e. the composition of $f$ with itself. Then there is a set $C^*\in\fil_{<\omega,\kappa}(\calA)$ such that $\forall \alpha\in C^*\forall \gamma\in \alpha\cap C^* (f^2(\gamma)<\alpha)$.  Now, let $C'=C \cap C^* \cap (f(1),\kappa)$. Thus $C'\in\fil_{<\omega,\kappa}(\calA)$. Let $\{k(\alpha): \alpha <\kappa\}$ be an increasing enumeration of $C' \cap \calA^h$. Since $C'\in\fil_{<\omega,\kappa}(\calA)$ the latter set is indeed unbounded in $\kappa$. Recursively, we will define a fusion sequence $\tau=\langle q_\alpha:\alpha\in\kappa\rangle$ below $p$ such that:

\begin{enumerate}
    \item $q_{\alpha+1} \leq_\alpha q_\alpha$,
    \item $q_{\alpha+1} \Vdash k(\alpha) \in \dot{X}$,
    \item If $q$ is the fusion of $\tau$ then $q\Vdash C'\cap \calA^h\subseteq\dot{X}$.
\end{enumerate}

But then, since $q\Vdash\calA^h\backslash\dot{X}\subseteq\calA^h\backslash C'$ and 
$\calA^h\backslash C'\subseteq\kappa\backslash C'\in\id_{<\omega,\kappa}(\calA)$, we obtain that $q\Vdash\calA^h\backslash\dot{X}\in\id_{<\omega,\kappa}(\calA)$, which is again a contradiction to the choice of $p$.

Here is the construction of $\tau$.
Start with $q_0 =p$ and at limits take intersections. Consider $k(0)$ and put $r = \stem(p)$. Since 
$$(C\cap\calA^h)\backslash f(0)\subseteq \bigcap_{t\in\hbox{split}_1(p)} Y_t,$$
for each $j \in\{0,1\}$ and $p$ is preprocessed for $\dot{X}$ there is $r_j(t,k(0))\in\split_{H(k(0))}(p)$ such that $p_{r_j(t,k(0))}\Vdash\check{k}(0)\in\dot{X}$. Let
$q_1=\bigcup\{p_{r_j(t,k(0))}: t\in\split_1(p) , j \in\{0,1\}\}$. Thus $q_1\leq_0 q_0$ as we wanted. 

For completeness we present the construction of the next step: take $k(1)$ and $t \in \split_1(q_1)$. Note that $\split_1(q_1)= \split_{\delta}(p)$ for some $1 \leq \delta \leq H(k(0))\leq f(k(0))$ and since $f^2(k(0)) < k(1)$ and $C \backslash Y_r \subseteq f(f(k(0)))$ for all $r \in \split_{f(k(0))+1}(p)$, we obtain $k(1) \in Y_r$ for all $r \in \split_{f(k(0))+1}(p)$. Using the fact that $p$ is preprocessed and repeating the argument above, find for each $j \in\{0,1\}$ an extension $r_j(t,k(1))\leq p_{t^\frown j}$ such that $r_j(t,k(1)) \Vdash k(1) \in \dot{X}$. Let $q_2 = \bigcup\{r_j(t,k(1)): t \in \split_1(q_1) \wedge j \in\{0,1\} \}$. Then $q_2$ is a condition, $q_2 \leq_1 q_1$ and $q_2 \Vdash k(1) \in \dot{X}$.

In general, suppose we have constructed $q_\alpha$ and consider $k(\alpha)$, and $t \in \split_\alpha(q_\alpha)$, then there is $\delta \geq \alpha$ so that $t \in \split_\delta (p)$ and $\delta \leq H(k(\alpha))$. Again, since $f^2(k(\alpha)) < k(\alpha+1)$ and $C \backslash Y_r \subseteq f(f(k(\alpha)))$ for all $r \in \split_{f(k(\alpha))+1}(p)$ we get that $k(\alpha) \in Y_t$, so again use that $p$ is preprocessed and repeat the argument above to find conditions $r_j(t,k(\alpha)) \leq p_{t^\smallfrown j}$ forcing $r_j(t,k(\alpha)) \Vdash k(\alpha) \in \dot{X}$ $j \in\{0,1\}$. Put $q_{\alpha+1} = \bigcup\{r_j(t,k(\alpha)): t \in \split_\alpha(q_\alpha) \wedge j \in\{0,1\} \}$. Then $q_{\alpha+1}$ is a condition, $q_{\alpha+1} \leq_{\alpha} q_\alpha$ and $q_{\alpha+1} \Vdash k(\alpha) \in \dot{X}$.

\end{proof}

\begin{thm}\label{main} The generic maximal independent family adjoined by $\bbP_\calU$ over a model $V_0$ of GCH remains maximal after the $\kappa$-support product $\mathbb{S}_\kappa^\lambda$.
\end{thm}

\begin{proof}
The idea for this proof is not much different than the one for the successor step, however the fusion argument has to be handled more carefully. We give now an outline with the most important details. As in the case above start with a densely maximal independent family $\calA$ in $V=V^{\mathbb{P}_\calU}$. To prove that $$(\calA\hbox{ is densely maximal independent family})^{V^{\mathbb{S}_\kappa^\lambda}}$$ we will show that in $V^{\mathbb{S}_\kappa^\lambda}$, property $(*)$ from Lemma~\ref{equiv} holds for the family $\calA$. Specifically, we show that in $V^{\mathbb{S}_\kappa^\lambda}$ for each $X\subseteq\kappa$ and each $h\in\FF_{<\omega,\kappa}(\calA)$ such that $X\subseteq\calA^h$ property $(*)_{X,h}$ holds, where
$(*)_{X,h}$ states:
$$\hbox{either }\exists B\in\id_{<\omega,\kappa}(\calA)\hbox{ such that }\calA^h\backslash X\subseteq B\hbox{ or there is }h'\supseteq h\hbox{ such that }\calA^{h'}\subseteq\calA^h\backslash X.$$

\noindent
Suppose not. Thus, there are $X\subseteq\kappa$, $h\in\FF_{<\omega,\kappa}(\calA)$ such that $X\subseteq \calA^h$ and $\lnot(*)_{X,h}$. That is,
$$V^{\mathbb{S}_\kappa^\lambda}\vDash X\subseteq\calA^h\land \calA^h\backslash X\notin\id_{<\omega,\kappa}(\calA)\land\forall h'\supseteq h(\calA^{h'}\cap X\neq\emptyset).$$
Let $\dot{X}$ be a $\mathbb{S}_\kappa^\lambda$-name for $X$ in $V$ and let $p\in\mathbb{S}_\kappa^\lambda$ force the above. Passing to a stronger condition if necessary we can assume that $p$ is preprocessed for $\dot{X}$. Now, for every $\alpha< \kappa$ and all $F \subseteq \supp(p)$ such that $\lvert F \lvert < \kappa$ consider the set $\Lambda^F_\alpha(p)$. This is a set of size $<\!\kappa$ because $\lvert \split_\al(p(i)) \lvert \leq 2^{\lvert \al \lvert}$ and $\lvert F \lvert < \ka$. Also, $\{\Lambda^F_\alpha(p): \alpha<\kappa \wedge F \subseteq \supp(p) \wedge \lvert F \lvert < \kappa \}$ has size $\ka$ and so does its union.


\medskip
Similarly to Claim \ref{Claim43} we obtain:
\begin{clm} For all $\alpha< \kappa$ and $F \subseteq \supp(p)$ such that $\lvert F \lvert <\kappa$, if $\bar{\sigma} \in \Lambda^F_\alpha(p)$ then $Y_{\bar{\sigma}}\subseteq \calA^h$.
\end{clm}

Before proceeding with construction of a fusion sequence, which will lead to the desired contradiction, we need to define one more auxiliary object, namely the function $H$ defined below. Whenever $\bar{\sigma} \in \Lambda^F_\alpha(p)$ and $\beta\in Y_{\bar{\sigma}}$, let $r(\bar{\sigma},\beta)$ be a witness to $\beta\in Y_{\bar{\sigma}}$ such that for each $i \in F $ the stem of the condition $r(\bar{\sigma},\beta)(i)$ is of minimal height. Define $H\in {^\kappa\kappa}\cap V$ as follows:
$$H(\gamma)=\sup\{\gamma+1\}\cup\{\Sl(\stem(r(\bar{\sigma},\beta))): \bar{\sigma} \in \Lambda_\gamma^F(p), \beta\leq \gamma, F\subseteq \supp(p)\cap \gamma \hbox{ and } \lvert F \lvert <\kappa \}.$$

Again, following the argument for the single step, we can assume that for all $\alpha<\kappa$ and all $\bar{\sigma} \in\Lambda_\alpha^F(p)$ there exists $B_{\bar{\sigma}}\in\id_{<\omega,\kappa}(\calA)$ such that $\calA^h\backslash Y_{\bar{\sigma}} \subseteq B_{\bar{\sigma}}$, $\calA^h\backslash Y_{\bar{\sigma}}\in\id_{<\omega,\kappa}(\calA)$. But then $Y_{\bar{\sigma}}\cup\kappa\backslash \calA^h\in\fil_{<\omega,\kappa}(\calA)$. Now, since $\fil_{<\omega,\kappa}(\calA)$ is a $\kappa$-p-set, there is $C\in\fil_{<\omega,\kappa}(\calA)$ such that $C\subseteq^* Y_{\bar{\sigma}}\cup\kappa\backslash\calA^h$ for each $\bar{\sigma} \in\Lambda_\alpha^F$. Thus in particular, $C\cap\calA^h\subseteq Y_{\bar{\sigma}}$ for each $\bar{\sigma} \in\Lambda_\alpha^F(p)$ and so we can find a function $f\in V\cap{^\ka\ka}$ such that,
$$\forall \alpha\in\kappa \,\,\forall F \subseteq \supp(p) \,\,\forall \bar{\sigma} \in \Lambda^F_{\alpha+1}(p) \,\,(\lvert F \lvert <\kappa \to C\cap (\calA^h\backslash Y_{\bar{\sigma}} )\subseteq f(\alpha))$$

Thus, for each $\alpha\in\kappa$, $\bar{\sigma} \in \Lambda^F_{\alpha+1}(p)$ and $\beta\in C\cap\calA^h$, if $\beta>f(\alpha)$ then $\beta\in Y_{\bar{\sigma}}$. Moreover, we can assume that $f$ is strictly increasing, that $H(\alpha)\leq f(\alpha)$ and $\alpha+2<f(\alpha)$ for all $\alpha\in\kappa$.

\medskip
By Corollary~\ref{closure_club}, there is $C^*\in\fil_{<\omega,\kappa}(\calA)$ such that $\forall \alpha\in C^*\forall \gamma\in \alpha\cap C^* (f^2(\gamma)<\alpha)$. Now, let $C'=C \cap C^* \cap (f(0),\kappa)$. Thus $C'\in\fil_{<\omega,\kappa}(\calA)$. Let
$\{k(\alpha): \alpha \in \kappa\}$ be an increasing enumeration of $C' \cap \calA^h$. Since $C'\in\fil_{<\omega,\kappa}(\calA)$ the latter set is indeed unbounded in $\kappa$.
Recursively, we will define a fusion sequence $\tau=\langle q_\alpha, F_\alpha:\alpha\in\kappa\rangle \subseteq \mathbb{S}_\kappa^\lambda$ below $p$, ordinals $(\eta_\alpha: \alpha< \kappa)$ and bijections $g_\alpha: F_\alpha \to \eta_\alpha$ such that for all $\alpha<\kappa$:

\begin{enumerate}
    \item $q_{\alpha+1} \leq_{F_\alpha,\alpha} q_\alpha$,
    \item $\eta_\alpha \geq \alpha$,
    \item For all $\alpha< \alpha'<\kappa$, $g_\alpha \subseteq g_{\alpha'}$ and for limit ordinals $\delta< \kappa$ $g_\delta = \bigcup_{\alpha< \delta} g_\alpha$,
    \item $q_{\alpha+1} \Vdash k(\alpha) \in \dot{X}$.
\end{enumerate}

\noindent
Now, if $q$ is the fusion of $\tau$ then $q\Vdash C'\cap \calA^h\subseteq\dot{X}$. But then, since $q\Vdash\calA^h\backslash\dot{X}\subseteq\calA^h\backslash C'$ and 
$\calA^h\backslash C'\subseteq\omega\backslash C'\in\id_{<\omega,\kappa}(\calA)$, we obtain that $q\Vdash\calA^h\backslash\dot{X}\in\id_{<\omega,\kappa}(\calA)$, contradicting the choice of $p$.

\medskip
We proceed with the recursive construction of $\tau$. Start with $q_0 = p$, $F_0= \min(\supp(p))$, $\eta_0 = 0$ and $g_0= \emptyset$. At limit stages $\delta<\kappa$ the construction of $q_\delta$ and $F_\delta$ it is already been determined so that the conditions in Definition~\ref{fusion} are fulfilled. Finally, let $\eta_\delta = \sup_{\alpha<\delta} \eta_\alpha$ and $g_\delta$ as above.
Consider $k(0)$ and put $\bar{\sigma} = \langle \stem(p(\min(\supp(p))))\rangle$. Since for all $\bar{\tau} \in \Lambda^{F_0}_{1}(p)$, $C \setminus Y_{\bar{\tau}} \subseteq f(0)$, $k(0) \geq f(0)$ and $p$ is preprocessed for $(\dot{X}, F_0)$ we have that $k(0) \in Y_{\bar{\tau}}$ for all $\bar{\tau} \in \Lambda^{F_0}_{1}(p)$ and so, by Corollary \ref{preprocessed_hull} for each $h \in ^{F_0} 2$ there exists a set $F_0' \supseteq F_0$, a $\bar{\tau} \sqsupseteq \bar{\sigma}$ and condition $r_h(\bar{\tau}_h,k(0)) \leq p^h_{\bar{\tau}_h}$ such that $r_h(\bar{\tau}_h,k(0)) \Vdash k(0) \in \dot{X}$. 

For completeness, we give a more detailed proof of the existence of the conditions $r_h(\bar{\tau}_h,k(0))$. Recall that $p^h_{\bar{\sigma}}$ is defined as follows:
\[p^h_{\bar{\sigma}}(i) =
\begin{cases}
(p(i))_{ \sigma_i^\smallfrown h(i)} & \text{if } i \in F \\
p(i)  & \text{ otherwise }
\end{cases}
\] 
There are sequences $\bar{\sigma}'_h \in \Lambda^{F_0}_1(p)$ such that $\bar{\sigma}_h \unlhd \bar{\sigma}'_h$ where $\bar{\sigma}_h$ is such that $\sigma_h(i)= \sigma(i)^\frown h(i)$ for all $i \in F_0$. Hence, using that $p$ is preprocessed we can get a set $F_0' \supseteq F_0$ such that $\lvert F_0' \lvert < \kappa$ and sequences $\bar{\tau}_h \sqsupseteq \bar{\sigma}'_h$ so that $p_{\bar{\tau}_h} \Vdash k(0) \in \dot{X}$. Thus, take $r_h(\bar{\tau}_h,k(0))=p_{\bar{\tau}_h}$.

Then if $q_1= \bigcup\{ r_h(\bar{\sigma},k(0)) : h \in {^{F_0}} 2, \bar{\tau}_h \in \Lambda^{F_0'}_0(p) \}$, $F_1= F_0' \cup \{\min\{\supp(q_1) \backslash F_0'\}\}$, $\eta_1= \eta_0 +1$ and $g_1 = g_0 \cup \{ (\min\{\supp(q_1) \backslash F_0', \eta_1)\}$ we have that 
$q_1 \leq_{0,F_0'} q_0\hbox{ and }q_1 \Vdash k(0) \in \dot{X}$ as we wanted. 

\medskip
In general, suppose we have constructed $q_\alpha$, $F_\alpha$, $\eta_\alpha$ and $g_\alpha$ as desired. Consider $k(\alpha)$ and the set $\Lambda^{F_\alpha}_{\alpha+1}(p_\alpha)$. 
Fix an enumeration $\{\gamma_l =(\bar{\sigma}_l, h_l): l< \rho \}$ of all pairs of the form $(\bar{\sigma}, h)$ such that $\bar{\sigma} \in \Lambda_{\alpha+1}^{F_\alpha}(q_\alpha)$ and $h \in {^{F_\alpha}2}$. Note that the ordinal $\rho$ is $<\kappa$. Inductively, we will construct a sequence $\{r^\alpha_l:l<\rho\}$ of conditions below $q_\alpha$ satisfying:

\begin{enumerate}
    \item $r^\alpha_0 = q_\alpha$,
    \item $r^\alpha_{l+1} \leq_{\alpha, F_\alpha} r^\alpha_{l}$,
    \item $(r^\alpha_{l+1})_{\bar{\sigma}_l}^{h_l} \Vdash \check{k}(\alpha) \in \dot{X}$, 
    \item For $l$ limit ordinal $r^\alpha_l = \bigwedge_{k< l} r^\alpha_k$.
\end{enumerate}

It is enough to explain how the successor step is built: Suppose then that we have constructed $r^\alpha_l$ satisfying the conditions above and consider the pair $\gamma_l =(\bar{\sigma}, h)$. Notice that since $\lvert F_\alpha \lvert < \kappa$ for all $i \in F_\alpha$, $\split_\al(q_\alpha(i)) \subseteq \split_\delta(p(i))$ for some $\delta \leq H(k(\alpha))$.  Also, since $f^2(k(\alpha)) < k(\alpha+1)$ and for all $\bar{\tau} \in \Lambda_{f(\alpha)+1}^{F_\alpha}(p)$ we have that $C \backslash Y_{\bar{\tau}} \subseteq f(f(k(\alpha)))$, we obtain that $k(\alpha) \in Y_{\bar{\tau}}$ for all $\bar{\tau} \in \Lambda_{f(\alpha)+1}^{F_\alpha}(p)$. Thus, again we repeat the argument above using the fact that $p$ is preprocessed for $(\dot{X}, F_\alpha)$ and find a set $F_\alpha' \supseteq F_\alpha$, a sequence $\bar{\tau}_h \in \Lambda^{F_\alpha'}_\alpha(p)$ and conditions $w_h(\bar{\tau}_h,k(\alpha)) \leq p^h_{\bar{\tau}_h}$ forcing $w_h(\bar{\tau}_h,k(\alpha)) \Vdash k(\alpha) \in \dot{X}$. Moreover, we can choose $w_h(\bar{\tau}_h,k(\alpha))$ so that $w_h(\bar{\tau}_h,k(\alpha))\leq (q_\alpha)^h_{\bar{\tau}_h}$. Note that $w_h(\bar{\tau}_h,k(\alpha))$ might not yet be a condition satisfying the condition (2). In order to fix this, we define $r^\alpha_{l+1}$ as follows: $\supp(r^\alpha_{l+1})=\supp(w_l)$ and 

\[r^\alpha_{l+1}(i) =
\begin{cases}
(w_h(\bar{\tau}_h,k(\alpha))(i) \cup \{ (q_\alpha(i))_{\rho ^\frown j}: \rho \in \split_\alpha (r_l(i)) \setminus \sigma_i \hbox{ or } j= 1-h(i)  \} & \text{if } i \in F_\alpha'  \\
 w_h(\bar{\sigma},k(\alpha))(i) & \text{ otherwise }
\end{cases}
\]
\medskip

Finally, let $q_{\alpha+1}= \bigwedge_{l< \rho} r^\alpha_l$,  $\eta_{\alpha+1} = \eta_\alpha +1$, $F_{\alpha+1}=F_\alpha'\cup\{\min(\supp(q_{\alpha+1})\backslash F_\alpha')\}$ and let $g_{\alpha+1} = g_\alpha \cup \{(\min(\supp(q_{\alpha+1})\backslash F_\alpha'), \eta_\alpha)\}$. The construction is now complete. Indeed, to see that $q_{\alpha+1} \Vdash \check{k}(\alpha) \in \dot{X}$ notice that for all $\bar{\sigma} \in \Lambda^{F_\alpha}_\alpha(q_{\alpha+1})$ and all $h \in {^{F_\alpha}} 2$, $(q_{\alpha+1})^h_{\bar{\sigma}} \Vdash \check{k}(\alpha) \in \dot{X}$.
\end{proof}

\begin{remark} Note that $\kappa$ might cease to be measurable in $V^{\mathbb{S}_\kappa^\lambda}$ from the above theorem. For a preparation of the universe, which guarantees that $\kappa$ remains measurable see~\cite{SF}.
\end{remark}

\section{Concluding Remarks and Questions}\label{questions_section}

The use of the assumption $2^\kappa=\kappa^+$ played a crucial role in our construction of a densely maximal $\kappa$-independent family. Thus one may ask:

\begin{question}
	Does ZFC imply the existence of a densely maximal $\kappa$-independent 
	families?
\end{question}

Even though we are able to show both that consistently $\mathfrak{i}_f(\kappa)=\kappa^+ < 2^\kappa$ and $\kappa^+<\mathfrak{i}_f(\kappa)=2^\kappa$, the currently available  techniques seem to be insufficient to answer the following:

\begin{question}\label{Mid_Size}
	Let $\kappa$ be a regular uncountable cardinal. Is it consistent that $\kappa^+<\mathfrak{i}(\kappa)<2^\kappa$?
\end{question}

The analogous question in the countable can be answered to the positive with the use of the so called diagonalization filters (see~\cite{VFSS}). A natural generalization of the notion of a diagonalization filter to the uncountable is given below:

\begin{defi}\label{diag_def}
	Let $\calA$ be a $\kappa$-independent family. A $\kappa$-complete filter $\calF$ is said to be an {\emph{$\kappa$-diagonalization filter for $\calA$}} if $\forall F\in\calF\forall h\in\FF_{<\omega,\kappa}(\calA)|F\cap\calA^h|=\kappa$
	and $\calF$ is maximal with respect to the above property. 
\end{defi}

Moreover, as a straightforward generalization of the countable case (see~\cite{VFSS}) one can show that:

\begin{lem}(see~\cite[Lemma 2]{VFSS}) Suppose $\calA$ is a $\kappa$-independent family and $\calF$ is a $\kappa$-diagonalization filter for $\calA$. Let $\bbM^\kappa_\calF$ be the generalized Mathias forcing relativized to the filter $\calF$.\footnote{That is $\bbM^\kappa_\calF$ consists of all pairs $(a,A)\in[\kappa]^{<\kappa}\times \calF$ such that $\sup a < \min A$.} Let $G$ be a $\bbM^\kappa_\calF$-generic filter and let $x_G=\bigcup\{a:\exists A (a,A)\in G\}$. Then $\calA\cup \{x_G\}$ is $\kappa$-independent and moreover for each $Y\in ([\kappa]^\kappa\cap V)\backslash\calA$ such that $\calA\cup\{Y\}$ is $\kappa$-independent, the family $\calA\cup\{x_G,Y\}$ is not $\kappa$-independent. 
\end{lem}

Even though an appropriate iteration of posets of the above form would produce a positive answer to Question~\ref{Mid_Size}, the following remains open:

\begin{question}
	Given a $\kappa$-independent family $\calA$ is there a  $\kappa$-diagonalization filter for $\calA$? The co-bounded filter does satisfy the characterization property in Definition~\ref{diag_def}, however the requirement for maximality is not straightforward to satisfy. 
	Is there a large cardinal property which guarantees the existence of such maximal filter? Note that a diagonalization filter is never an ultrafilter.
\end{question}

Moreover of interest remain the following:

\begin{question}
	Is it consistent that $\mathfrak{i}(\kappa)<\mathfrak{a}(\kappa)$?
\end{question}

Clearly, if the above is consistent then in the corresponding model,  $\mathfrak{i}(\kappa)\geq\kappa^{++}$. One of the original questions, which motivated the work on this project is the evaluation of $\mathfrak{i}(\kappa)$ in the model from~\cite{BTVFSFDM}. More precisely, we would like to know:

\begin{question}
	Is it consistent that $\mathfrak{i}(\kappa) < \mathfrak{u}(\kappa)?$
\end{question}

The consistency of $\mathfrak{r}<\mathfrak{i}$ holds in the Miller model. However, products of the generalized Miller poset $\mathbb{MI}_\kappa^\calU$,
where $\calU$ is a $\kappa$-complete normal ultrafilter on $\kappa$ add $\kappa$-Cohen reals (see~\cite[Theorem 85]{BBTFM}) and so increase $\mathfrak{r}(\kappa)$. Even though 
$\mathbb{MI}_\kappa^\calU$ has the generalized Laver property (see~\cite[Proposition 81]{BBTFM}), it is open if the generalized Laver property is preserved under $\kappa$-support iterations. This leaves us with the following:

\begin{question}
	Is it consistent that $\mathfrak{r}(\kappa)<\mathfrak{i}(\kappa)$?
\end{question}

\section{Appendix: Strong Independence}\label{appendix_section}

Another approach towards finding a higher analogues of independence for a given uncountable cardinal $\kappa$ is to consider boolean combinations generated by strictly less than $\kappa$ (not just finitely) many members of the family. More precisely one can give the following definition:

\begin{defi} Let $\kappa$ be a regular uncountable cardinal, $\calA\subseteq [\kappa]^\kappa$ of cardinality at least $\kappa$.
	\begin{enumerate}
		\item Let $\FF_{<\kappa,\kappa}(\calA)$ be the set of
partial functions $h:\calA\to\{0,1\}$ with domain of cardinality strictly below $\kappa$ and for $h\in\FF_{<\kappa,\kappa}(\calA)$ let $\calA^h=\bigcap\{\calA^{h(A)}: A\in\dom(h)\}$ where $\calA^{h(A)}=A$ if $h(A)=0$ and $\calA^{h(A)}=\kappa\backslash A$ if $h(A)=1$. 
	    \item The family $\calA$ is said to be strongly-$\kappa$-independent 
	    if for every $h\in \FF_{<\kappa,\kappa}(\calA)$ the boolean combination $\calA^h$ is unbounded.  
	    \item The family $\calA$ is said to be maximal strongly-$\kappa$-independent if it is strongly-$\kappa$-independent and is not properly contained in another strongly-$\kappa$-independent family. 
	    \item Suppose $\kappa$ is a regular uncountable cardinal for which maximal strongly-$\kappa$-independent families exists. With  $\mathfrak{i}_s(\kappa)$ we denote the minimal size of a maximal strongly-$\kappa$-independent family. 
\end{enumerate}
\end{defi}

Note that the increasing union of a countable sequence of strongly-$\kappa$-independent families is not necessarily strongly-$\kappa$-independent. Thus one can not apply Zorn's lemma to claim the existence of maximal strongly-$\kappa$-independent families. What we can say is the following:

\begin{thm}\label{strong_classic} Let $\kappa$ be a regular uncountable caridnal.
	\begin{enumerate}
 \item For $\kappa$ strongly inaccessible, there is a strongly-$\kappa$-independent family of cardinality $2^\kappa$.
 \item If $\calA$ is strongly-$\kappa$-independent and $|\calA|<\mathfrak{r}(\kappa)$ then $\calA$ is not maximal.
 \item Suppose $\mathfrak{d}(\kappa)$ is such that for every $\gamma<\mathfrak{d}(\kappa)$, $\gamma^{<\kappa}<\mathfrak{d}(\kappa)$. If $\calA$ is strongly-$\kappa$-independent and $|\calA|<\mathfrak{d}(\kappa)$ then $\calA$ is not maximal.
    \end{enumerate}
\end{thm}
\begin{proof}
	We will prove $(1)$. Let $\calC=\{(\gamma, A): \gamma<\kappa, A\subseteq\calP(\gamma)\}$. Given $X\subseteq \kappa$ define $\calY_X=\{(\gamma, A)\in \calC: X\cap \gamma\in A\}$. Then $\calY_X: X\subseteq\kappa\}$ is strongly-$\kappa$-independent. Indeed. Consider two disjoint subfamilies of $[\kappa]^\kappa$, each of size strictly smaller than $\kappa$, say $\{X_i\}_{i\in I_1}$ and $\{Z_j\}_{j\in I_2}$.
	Note that $(\gamma, A)\in \calX=\bigcap_{i\in I_1} \calY_{X_i}\cap\bigcap_{j\in I_2}(\calC\backslash \calY_{Z_j})$ if for all $i\in I_1$, $X_i\cap A\in A$ and for all $j\in I_2$, $Z_j\cap \gamma\notin A$. However, there are unboundedly many $\gamma\in\kappa$ such that 
	\begin{itemize}
		\item $X_i\cap \gamma\neq X_{i'}\cap \gamma$ for $i\neq i'$ both in $I_1$, and
		\item $Z_j\cap \gamma\neq Z_{j'}\cap\gamma$ for $j\neq j'$ both in $I_2$, and
		\item $X_i\cap \gamma\neq Z_j\cap \gamma$ for all $i\in I_1$, $j\in I_2$.
	\end{itemize} 
	It remains to observe that for each such $\gamma$, we have $(\gamma, A_\gamma)\in \calX$, where $A_\gamma=\{X_i\cap \gamma:i\in I_1\}$.

	To see part $(2)$ note that if $|\calA|<\mathfrak{r}(\kappa)$, then the set $\{\calA^h:h\in\FF_{<\kappa,\kappa}(\calA)\}$ is split by some $X\in[\kappa]^\kappa$ and so $\calA\cup\{X\}$ is strongly $\kappa$-independent which properly contains $\calA$.
	
	For a proof of part $(3)$, see~\cite[Proposition 27]{BTVFSFDM}.
\end{proof}

\begin{cor} Thus, if $\mathfrak{i}_s(\kappa)$ is defined, then $\kappa^+\leq\mathfrak{i}_s(\kappa)\leq 2^\kappa$. Moreover  $\mathfrak{r}(\kappa)\leq\mathfrak{i}_s(\kappa)$  and 
	if for every $\gamma<\mathfrak{d}(\kappa)$, $\gamma^{<\kappa}<\mathfrak{d}(\kappa)$, then $\mathfrak{d}(\kappa)\leq\mathfrak{i}_s(\kappa)$. 
\end{cor}

\begin{question} $ $
	\begin{enumerate}
	 \item Is there a large cardinal property which implies the existence of a maximal strongly-$\kappa$-independent family?
	 \item Given a strongly $\kappa$-independent family $\calA$, is there a large cardinal property which implies the existence of a $\kappa$-diagonalization filter for $\calA$?
     \item Suppose $\mathfrak{i}_s(\kappa)$ is defined. A family which is strongly-$\kappa$-independent is $\kappa$-independent. However a maximal strongly independent family is not necessarily maximal independent. Is there a ZFC relation between $\mathfrak{i}_s(\kappa)$ and $\mathfrak{i}(\kappa)$? 
\end{enumerate}
\end{question}

\bigskip

\noindent {\em Acknowledgements}. The authors would like to thank the anonymous referee for many helpful comments which significantly improved the quality of the manuscript.

\end{document}